\documentclass[11pt,twoside]{amsart}

\usepackage[english]{babel}
\usepackage[T1]{fontenc}
\usepackage{lmodern}
\usepackage{amsmath,amsthm,amssymb,amsfonts}
\usepackage{hyperref}
\usepackage[capitalise]{cleveref}
\usepackage[left=3cm,right=3cm,top=3cm,bottom=3cm,includeheadfoot]{geometry}
\usepackage{bbm}
\usepackage{xcolor}
\usepackage{enumerate,enumitem}
\usepackage{fancyhdr}
\usepackage{tikz-cd}
\usepackage[abbrev,nobysame]{amsrefs}
\usepackage{graphicx}
\usepackage{subcaption}

\usepackage{comment}

\hypersetup{colorlinks=true,
	linkcolor=teal,
	citecolor=orange}

\theoremstyle{plain}
\newtheorem{lem}{Lemma}[section]
\newtheorem{prop}[lem]{Proposition}
\newtheorem{thm}[lem]{Theorem}
\newtheorem{cor}[lem]{Corollary}

\AddToHook{env/exl/begin}{\crefalias{lemma}{exl}}
\AddToHook{env/thm/begin}{\crefalias{lemma}{thm}}

\theoremstyle{definition}
\newtheorem{defi}[lem]{Definition}
\newtheorem{exl}[lem]{Example}

\theoremstyle{plain}
\newtheorem{thmintro}{Theorem}

\newtheorem*{probintro*}{Problem}

\newtheorem*{corintro*}{Corollary}

\numberwithin{equation}{section}

\crefname{enumi}{}{}
\crefname{equation}{}{}

\crefname{exl}{Example}{Examples}

\crefname{thm}{Theorem}{Theorems}

\newcommand{\R}{\mathbb{R}}			
\renewcommand{\S}{\mathbb{S}}		
\newcommand{\Gr}{\mathrm{Gr}}		
\newcommand{\DD}{\mathbb{D}}		

\newcommand{\indf}{\mathbbm{1}}		

\newcommand{\K}{\mathcal{K}}		

\DeclareMathOperator{\SO}{SO}		
\DeclareMathOperator{\OO}{O}		

\newcommand{\abs}[1]{\lvert #1 \rvert}						
\newcommand{\norm}[1]{\lVert #1 \rVert}						
\newcommand{\pair}[2]{\left\langle #1,#2 \right\rangle}				

\BibSpec{article}{%
	+{}{\PrintAuthors}  		{author}
	+{,}{ \textit}     		{title}
	+{,}{ }             		{journal}
	+{}{ \textbf}       		{volume}
	+{}{ \parenthesize} 		{date}
	+{}{, no. } 		{number}
	+{,}{ }      	      		{conference}
	+{,}{ }      	      		{book}
	+{,}{ }            		{pages}
	+{,}{ }            	 	{note}
	+{,}{ }            	 	{status}
	+{,}{  \texttt } {eprint}
	+{.}{}              {transition}
}

\BibSpec{book}{%
	+{}{\PrintAuthors}  {author}
	+{,}{ \textit}      {title}
	+{,}{ }             {publisher}
	+{,}{ }             {place}
	+{,}{ }             {date}
	+{.}{}              {transition}
}

\BibSpec{incollection}{%
	+{}  {\PrintAuthors}                {author}
	+{,} { \textit}                     {title}
	+{.} { }                            {part}
	+{:} { \textit}                     {subtitle}
	+{,} { \PrintContributions}         {contribution}
	+{,} { \PrintConference}            {conference}
	+{,} { }                            {booktitle}
	+{,} { pp.~}                        {pages}
	+{,}{ }       		{series}
	+{}{ \textbf}       		{volume}
	+{,} { }                            {publisher}
	+{,} { }                 {date}
	+{,} { }                            {status}
	+{,} { \PrintDOI}                   {doi}
	+{,} { available at \eprint}        {eprint}
	+{}  { \parenthesize}               {language}
	+{}  { \PrintTranslation}           {translation}
	+{;} { \PrintReprint}               {reprint}
	+{.} { }                            {note}
	+{.} {}                             {transition}
}

\usepackage{xcolor}
\begin{document}
	\title[The Christoffel problem for the disk area measure]{The Christoffel problem\\ for the disk area measure}
	\author{Leo Brauner}
	\address{Institut f\"ur Mathematik \newline%
		\indent Goethe-Universit\"at Frankfurt \newline%
		\indent Robert-Mayer-Str. 10, 60325 Frankfurt am Main, Germany}
	\email{brauner@math.uni-frankfurt.de}
	
	\author{Georg C.\ Hofst\"atter}
	\address{Institut f\"ur diskrete Mathematik und Geometrie \newline%
		\indent Technische Universit\"at Wien \newline%
		\indent Wiedner Hauptstrasse 8-10/1046, 1040 Wien, Austria}
	\email{georg.hofstaetter@tuwien.ac.at}
	
	\author{Oscar Ortega-Moreno}
	\address{Departamento de Matem\'aticas
        \newline%
		\indent CUNEF Universidad
        \newline%
		\indent Madrid, Spain}
	\email{oscar.ortegamoreno@cunef.edu}
	
	\begin{abstract}
		The mixed Christoffel problem asks for necessary and sufficient conditions for a Borel measure on the Euclidean unit sphere to be the mixed area measure of some convex bodies, all but one of them are fixed.
        We consider the case in which the reference bodies are $(n-1)$-dimensional disks lying in a fixed hyperplane. We obtain an integral representation that reconstructs the support function of a convex body from its disk area measure, without any regularity assumptions. In the smooth setting, we reformulate the problem as a linear differential equation on the sphere, and derive a necessary and sufficient condition on the density of the disk area measure guaranteeing both convexity and regularity of the solution.
	\end{abstract}
	
	\maketitle

	\thispagestyle{empty}

	\section{Introduction}
	\label{sec:intro}
	For a convex body (convex, compact set) $K \subset \R^n$ with smooth boundary, the \emph{first-order area measure}, denoted $S_1(K, \cdot)$, is the absolutely continuous measure on the unit sphere $\S^{n-1}$ with density given by the mean radius of curvature of $K$ as a function of the outer unit normal vector. The definition of $S_1(K, \cdot)$ can be extended to general convex bodies via variations of the \emph{surface area measure} $S_{n-1}(K, \cdot)$, that is,
	\begin{equation}\label{eq:area_meas_isotropic_def}
		S_1(K,\beta)
		= \frac{1}{(n-1)!}\left.\left(\frac{d}{dt}\right)^{\! n-2}\right|_{t=0^+} S_{n-1}(K + tB^n,\beta),
	\end{equation}
	for every Borel set $\beta \subseteq \S^{n-1}$, where $B^n$ denotes the Euclidean unit ball. Here, the surface area measure $S_{n-1}(K,\beta)$ evaluated at a Borel set $\beta$ is given by the Hausdorff measure $\mathcal{H}^{n-1}$ of all boundary points of $K$ with outer unit normal in $\beta$. We refer to \cite{Schneider2014}*{Sec.~4} for a more detailed exposition.
	
	Dating back to the nineteenth century with the work of Christoffel~\cite{Christoffel1865}, it was an important question in convex geometry, nowadays known as the \emph{Christoffel problem}, to determine the class of Borel measures on the sphere that arise as $S_1(K, \cdot)$ for some convex body $K$. After Christoffel's work \cite{Christoffel1865}, treating the three dimensional case, a complete solution was obtained at the end of the 1960s with the works of Berg~\cite{Berg1969} and Firey~\cites{Firey1967,Firey1968} (see also \cite{Li2021} for a more recent approach). Special cases were treated for smooth bodies by Pogorelov~\cite{Pogorelov1953} and for polytopes by Schneider~\cite{Schneider1977b}.
	
	The key to solving the Christoffel problem was the observation that in the sense of distributions on $\S^{n-1}$, the area measure $S_1(K, \cdot)$ can be obtained from the support function $h_K$ of $K$ by means of a differential operator: 
	\begin{align}\label{eq:area_meas_PDE}
		S_1(K, {}\cdot{}) 
        = \frac{1}{n-1} \Delta_{\S}h_K + h_K,
	\end{align}
     where $\Delta_{\S}$ denotes the Laplace--Beltrami operator on $\S^{n-1}$. Using spherical harmonics, Berg~\cite{Berg1969} constructed a family of functions $g_n\in C^\infty(-1,1)$, $n \geq 2$, that serve as Green functions of the operator. Here and in the following, we denote by $\K(\R^n)$ the set of convex bodies in $\R^n$, by $s(K)$ the \emph{Steiner point} of a body $K\in\K(\R^n)$ (see \cref{sec:preliminaries}), and call a Borel measure $\mu$ on $\S^{n-1}$ \emph{centered} if $\int_{\S^{n-1}} u \, \mu(du) = o$. 
     Berg's result~\cite{Berg1969}*{Thm.~5.3} can be stated as follows.
     
	\begin{thm}[{\cite{Berg1969}}]\label{thm:Berg_Christoffel}
		Let $\mu$ be a centered, non-negative, finite Borel measure on $\S^{n-1}$. Then there exists a convex body $K\in\K(\R^n)$ with $\mu=S_1(K,{{}\cdot{}})$ if and only if
		\begin{equation}\label{eq:Berg_Christoffel}
			h(u)=\int_{\S^{n-1}} g_n(\pair{u}{v})\, \mu(dv),\qquad u\in\S^{n-1},
		\end{equation}
		is a support function. In that case $h=h_{K-s(K)}$.
	\end{thm}
	
	Recently, more general versions of the Christoffel problem, involving mixed area measures, have attracted considerable attention. Mixed area measures are closely related to mixed volumes of convex bodies, both of which are fundamental notions in convex geometry (see, e.g.~\cite{Schneider2014}). The first-order area measure $S_1(K,\cdot)$ of a body $K\in\K(\R^n)$ is an instance of this.
    Replacing the unit ball $B^n$ in \cref{eq:area_meas_isotropic_def} by an arbitrary convex body $C \in \K(\R^n)$, we obtain a rich family of mixed area measures:
	\begin{equation}\label{eq:area_meas_anisotropic_def}
		S_1(K,C;\beta)
		= \frac{1}{(n-1)!}\left.\left(\frac{d}{dt}\right)^{\! n-2}\right|_{t=0^+} S_{n-1}(K + tC,\beta)
	\end{equation}
	for every Borel set $\beta \subseteq \S^{n-1}$. In this article, we consider the following Christoffel type problem. 
	
	\begin{probintro*}
		Given $C\in \K(\R^n)$, find necessary and sufficient conditions for a Borel measure $\mu$ on $\S^{n-1}$ such that $\mu = S_1(K,C; {}\cdot{})$ for some convex body $K \in \K(\R^n)$.
	\end{probintro*}

    This is an instance of the \emph{mixed Christoffel problem}, where the role of the reference body $C$ is taken by a collection of reference bodies $C_1,\ldots,C_{n-2}$. For bodies that are symmetric about an axis, this problem has been completely solved by the authors of this article~\cite{Brauner2025a}, in the spirit of the classical work of Firey~\cite{Firey1970}. The corresponding problem for radially symmetric convex functions and mixed Monge--Ampère measures was treated by Mussnig--Ulivelli~\cite{Mussnig2025}. The two approaches are independent of each other, although both heavily exploit the imposed symmetry.

    In the general setting, recently, a significant advancement was achieved by Colesanti--Focardi--Guan--Salani~\cite{Colesanti2025}, who identified sufficient conditions characterizing a rich class of mixed area measures. Their approach is based on the techniques of Guan--Ma~\cite{Guan2003} and relies on smoothness assumptions. To date, the only instance of the mixed Christoffel problem where a complete solution is known, including necessary and sufficient conditions, remains the classical Christoffel problem (where $C=B^n$). While the solution of Berg and Firey is somewhat formal, the work of Li--Wan--Wang~\cite{Li2021} provides a more readily checkable condition via the fundamental solution to the Laplace equation.

    \subsection*{Main results}

    The aim of this article is to fully resolve another instance of the mixed Christoffel problem, namely where $C$ is the $(n-1)$-dimensional unit disk perpendicular to the $n$-th coordinate vector $e_n$, denoted as~$\DD$. Such \emph{disk area measures} involving the disk~$\DD$ turn out to possess particularly nice geometric properties, and have proved useful in a variety of applications (see \cites{Brauner2024a,Brauner2025a,Hug2024,Hug2024a,Mussnig2025,Colesanti2022a}). In particular, for axially symmetric bodies, the authors of this article reduce the mixed Christoffel problem (and more generally, the mixed Christoffel--Minkowski problem) to the one involving the disk area measure~\cite{Brauner2025a}.
    
    Our contribution to this problem is twofold. First, in the same spirit as Berg and Firey, we establish an integral representation expressing the support function of $K$ in terms of the disk area measure~$S_1(K,\DD;\cdot)$, without any regularity assumptions on $K$. Second, we treat the problem as a differential equation and derive an efficient condition on the density of the disk area measure to determine both the convexity and regularity of the solution.   
	
	In order to state our first main result, we let $\Gr_2(\R^n, e_n)$ denote the Grassmanian of all two-dimensional linear subspaces of $\R^n$ containing the $n$-th coordinate vector $e_n$, and let $\pi: \S^{n-1} \setminus\{\pm e_n\} \to \Gr_{2}(\R^n, e_n)$ be the map that assigns to every $u \in \S^{n-1}\setminus\{\pm e_n\}$ the unique plane $e_n \vee u$ spanned by $u$ and $e_n$.
	Given a Borel measure $\mu$ on $\S^{n-1}$, if its pushforward $\pi_\ast \mu$ is absolutely continuous with continuous density, then the measure $\mu$ disintegrates into a family $(\mu_E)_{E \in \Gr_{2}(\R^n, e_n)}$ of measures on $\S^1(E)$ such that for any function $f\in C(\S^{n-1})$,
	\begin{align}\label{eq:disintIntro}
		\int_{\S^{n-1}\setminus \{\pm e_n\}} f(u) \, \mu(du) = \int_{\Gr_2(\R^n, e_n)} \int_{\S^1(E)\setminus\{\pm e_n\}} f(v) \, \mu_E(dv)\, dE.
	\end{align}
	Note that, as long as $\mu(\{\pm e_n\}) = 0$, we may identify $\mu$ with its restriction to $\S^{n-1}\setminus\{\pm e_n\}$; see \cref{sec:disint} for the details on the construction of this disintegration. 
    
	\begin{thmintro}\label{mthm:Christoffel_disk_gnrl}
		Let $\mu$ be a non-negative, centered, finite Borel measure on $\S^{n-1}$ and suppose that $\mu(\{\pm e_n\})=0$.
		Then there exists a convex body $K\in\K(\R^n)$ with $\mu=S_1(K,\DD,{}\cdot{})$ if and only if
		\begin{enumerate}[label=\upshape(\roman*)]
			\item\label{it:ChristDiskAbsCont} The measure $\pi_\ast \mu$ is absolutely continuous with a continuous density,
			\item\label{it:ChristDiskCent}  $\int_{\S^1(E)}v\,\mu_E(dv)=0$ for a.e.\ $E\in\Gr_2(\R^n,e_n)$, and
			\item\label{it:ChristDiskIntCond} there exists a support function $h$ of a convex body such that
			\begin{equation}\label{eq:Christoffel_disk_gnrl}
				h(u)- \frac{1}{\pi}\int_{\S^1(e_n\vee u)} h(v)\pair{u}{v}\, dv = \int_{\S^1(e_n\vee u)}\sqrt{1-\pair{u}{v}^2}(\pi - \arccos\pair{u}{v}) \,\mu_{e_n\vee u}(dv),			
			\end{equation}
			for a.e.\ $u\in\S^{n-1}$.
		\end{enumerate}

		In this case, $K$ is unique up to a translation, and $h = h_{K-v}$ for some $v \in \R^n$.
	\end{thmintro}
	
	Let us comment on the condition that the measure has no mass at the poles $\pm e_n$. This is due to the fact that the solution of the Christoffel problem for the disk is not unique as soon as there is a face in either direction $\pm e_n$. Indeed, whenever $K \in \K(e_n^\perp)$, its disk area measure $S_1(K, \DD, \cdot) = V^{e_n^\perp}(K, \DD^{[n-2]})(\delta_{e_n} + \delta_{-e_n})$ only depends on the mean width of $K$.
	
	On a technical level, mass at the poles $\pm e_n$ raises the problem of how to define a disintegration similar to \cref{eq:disintIntro}. Indeed, the domains of the measures $\mu_E$ overlap at the poles $\pm e_n$, so a procedure is required to distribute the mass $\mu(\{\pm e_n\})$ among the measures $\mu_E$. However, this amounts to determining the (highly non-unique) geometry of the faces of $K$ in the directions $\pm e_n$. As soon as a disintegration satisfying \cref{eq:Christoffel_disk_gnrl} is constructed, the steps in the proof of \cref{mthm:Christoffel_disk_gnrl} can be repeated to yield the solution.

    Let us emphasize that \cref{mthm:Christoffel_disk_gnrl} provides a set of conditions that are necessary and sufficient without requiring any smoothness of $K$, and that it allows to retrieve the support function of $K$ from its disk area measure. However, as in the work of Berg and Firey, it does not offer a practical way to check convexity or to determine the regularity of $K$ depending on the regularity of~$\mu$. Hence we also pursue a second approach, where we consider the problem as a linear differential equation, similar to~\cref{eq:area_meas_PDE}.
    
    If we assume $K\in\K(\R^n)$ to be of class $C^2_+$ (that is, its boundary is a $C^2$~submanifold of $\R^n$ with everywhere positive Gauss curvature), then its support function $h_K$ is a $C^2(\S^{n-1})$ function. Parametrizing $\S^{n-1}$ in polar coordinates $u=\cos\theta e_n + \sin\theta w$, where $\theta\in [0,\pi]$ and $w\in\S^{n-2} = \S^{n-2}(e_n^\perp)$, the disk area measure $S_1(K,\DD,{}\cdot{})$ can be expressed in terms of $h_K$ as
    \begin{equation}\label{eq:disk_area_meas_PDE}
        S_1(K,\mathbb{D},du)
        = \frac{1}{n-1}\big(\partial_\theta^2 h_K(\theta,w) + h_K(\theta,w)\big) d\theta\, dw.
    \end{equation}
    Notably, the lower dimensionality of the disk $\DD$ causes the associated partial differential equation to degenerate into a family of ordinary differential equations along the meridians of the unit sphere, allowing us to employ standard theory of ordinary differential equations. 
    

    As in \cref{mthm:Christoffel_disk_gnrl}, the core of the argument consists in gluing together the solutions constructed on the meridians to obtain a global convex solution on $\S^{n-1}$. The advantage of working in the class $C^2_+$ is that the convexity of the resulting solution $K$ can be encoded as the positive definiteness of the Hessian of its support function. Our second main result provides a practical condition that fully characterizes disk area measures of $C^2_+$~regular convex bodies.
    
	\begin{thmintro}\label{mthm:Christoffel_disk_C2+}
        Let $q\in C([0,\pi]\times \S^{n-2})$ be strictly positive and $q(\theta,\,\cdot\,)\in C^2(\S^{n-2})$ for all $\theta \in [0,\pi]$.
        There exists some convex body $K\in\K(\R^n)$ of class $C^2_+$ such that $q(\theta,w)d\theta\, dw=S_1(K,\DD,du)$ if and only if there exist $t\in\R$ and $x \in\R^{n-1}$ such that for all $w\in\S^{n-2}$,
        \begin{equation*}
            \int_{0}^{\pi} \sin\theta \, q(\theta,w) \,d\theta = t \qquad \text{and} \qquad 
            \int_{0}^\pi \cos\theta \, q(\theta, w) \,d\theta = \langle x, w\rangle.
    \end{equation*}
    and for all $(\theta,w)\in (0,\pi) \times \S^{n-2}$ and unit tangent vectors $\xi \in T_w \S^{n-2}$,
        \begin{align*}
        \begin{split}
            &\sin^2\!\theta\, q(\theta,w)\int_{0}^\theta[\sin(\theta - \sigma)\xi^\top\nabla^2_{\S^{n-2}}q(\sigma,w) \,\xi  +  \sin\theta\cos\sigma q(\sigma,w) ] \,d\sigma \\
            &\qquad > \bigg(\int_{0}^\theta\sin\sigma\, \xi^{\top}\nabla^{\S^{n-2}}q(\sigma,w)\,d\sigma\bigg)^{\! 2}\! .
            \end{split}
        \end{align*}
        In this case, the body $K$ is unique up to translations, and $h_{K-v}=h$ for some $v\in\R^n$, where
        \begin{equation*}
            h(\theta,w) = (n-1)\int_{0}^\theta \sin(\theta - \sigma)q(\sigma, w) \, d\sigma , \qquad (\theta,w)\in [0,\pi]\times\S^{n-2}.
        \end{equation*}
    \end{thmintro}
    

    \subsection*{Plan of the article} \cref{mthm:Christoffel_disk_gnrl} is proved in \cref{sec:convol_rep} and \cref{mthm:Christoffel_disk_C2+} is proved in \cref{sec:convCritC2p}. The necessary background is given wherever it is needed first.
 
	\section{A convolution representation}\label{sec:convol_rep}

	Our proof of \cref{mthm:Christoffel_disk_gnrl} is inspired by the the Kubota-type formula of Hug, Mussnig, and Ulivelli~\cite{Hug2024} in the instance where $i=1$: For a convex body $K\in\K(\R^n)$ and a bounded Borel function $f:\S^{n-1}\to\R$,
	\begin{equation}\label{eq:zonal_Kubota_polarized_i=1}
		\int_{\S^{n-1}}  f(u)\, S_{1}(K,\DD,du)
		= \frac{\kappa_{n-1}}{2}  \int_{\Gr_{2}(\R^n,e_n)} \!\int_{\S^{1}(E)} \! f(u)\, S^E_{1}(K|E,du) \, dE.
	\end{equation}
	This integral geometric formula expresses the mixed area measure $S_1(K,\DD,{}\cdot{})$ as an average of surface area measures of two-dimensional projections of $K$. The core idea of our argument is to compare \cref{eq:zonal_Kubota_polarized_i=1} to the disintegration \cref{eq:disintIntro} of the measure $\mu$ to relate the mixed Christoffel problem with the disk to the classical Christoffel problem in two dimensions.
	
	By \cref{thm:Berg_Christoffel}, we explicitly solve the $2$-dimensional Christoffel problem for the measures $\mu_E$, giving bodies $K_E \in \K(E)$, $E \in \Gr_2(\R^n, e_n)$. This is where the integral kernel
    \begin{align*}
		g_2(t) = \sqrt{1-t^2}(\pi - \arccos t) + c_n t,
	\end{align*}
	originates, with $c_n\in\R $ chosen in such a way to make the function $u \mapsto g_2(\pair{e_n}{u})$ centered on $\S^{n-1}$.
    If there is a body $K \in \K(\R^n)$ such that every $K_E$ arises as the orthogonal projection $K|E$ of $K$ onto $E$, then we can conclude $\mu = S_1(K,\DD; {{}\cdot{}})$. However, as the subspaces $E$ intersect only in the line spanned by $e_n$, compatibility of the $K_E$ is no issue; convexity is asserted by condition \cref{it:ChristDiskIntCond} in \cref{mthm:Christoffel_disk_gnrl}.
	
	Let us point here that a similar strategy is promising for the intermediate Christoffel--Minkowski problems with the disk as reference body, replacing \cref{eq:zonal_Kubota_polarized_i=1} by the $i$-homogeneous counterpart from \cite{Hug2024}. However, in this case, the subspaces do intersect, which directly implies convexity but makes compatibility an issue.
	
	\medskip
	
	In the following, we first recall some classical facts from convex geometry, in particular on mixed volumes and mixed area measures that will be needed later on. Then we show the existence of the disintegration \cref{eq:disintIntro}. In the last section we finally prove \cref{mthm:Christoffel_disk_gnrl}.
	
	\subsection{Preliminaries}\label{sec:preliminaries}
	As a general reference on this section, we refer to the monographs by Gardner~\cite{Gardner2006} and Schneider~\cite{Schneider2014}. 

    First, the \emph{Steiner point} (see, e.g., \cite{Schneider2014}*{p.~50}) of a convex body $K \in \K(\R^n)$ is defined as
    \begin{equation}\label{eq:Steiner_point}
		s(K)
		= \frac{1}{\kappa_n}\int_{\S^{n-1}} h_K(u)u\, du\in\R^n.
	\end{equation}
    
	Next, the mixed volume $V(K_1, \dots, K_n)$, $K_1, \dots, K_n \in \K(\R^n)$ is defined as the suitable coefficient of the homogeneous polynomial obtained by Minkowski addition
	\begin{align*}
		V(\lambda_1 K_1 + \dots + \lambda_n K_n) = \sum_{j_1,\dots,j_{n}=1}^n \lambda_{j_1}\cdots\lambda_{j_n} V(K_1, \dots, K_n), \qquad \lambda_1, \dots, \lambda_n \geq 0.
	\end{align*}
	The mixed volume can be represented (or localized) using the Riesz representation theorem,
	\begin{equation}\label{eq:MixedVol_Integral}
		V(L,K_1,\ldots,K_{n-1})
		= \frac{1}{n} \int_{\S^{n-1}} h_{L}(u)\, S(K_1,\ldots,K_{n-1};du), \quad L \in \K(\R^n),
	\end{equation}
	defining a centered, non-negative Borel measure $S(K_1, \ldots, K_{n-1};{{}\cdot{}})$ on $\S^{n-1}$, called the mixed area measure of $K_1, \dots, K_{n-1} \in \K(\R^n)$. Moreover, by definition, it is $1$-homogeneous and translation-invariant in every component, as well as symmetric under permuting the entries. For $K_1 = \dots = K_i = K$ and $K_{i+1} = \dots = K_{n-1} = B^n$, $0\leq i \leq n$, and re-normalizing, we obtain the $i$th intrinsic volume of $K \in \K(\R^n)$,
	\begin{equation}\label{eq:MixedVol_IntrinsicVol}
		V_i(K)
		= \frac{\binom{n}{i}}{\kappa_{n-i}}V(K^{[i]},(B^n)^{[n-i]}).
	\end{equation}
	Here, we denote by $K^{[i]}$ the $i$-tuple $(K, \dots, K)$ with $K$ repeated $i$-times.

	If all the bodies $K_1\ldots,K_{n-1}$ are contained in the hyperplane $u^\perp$, then
	\begin{equation}\label{eq:MixedAreaMeas_hyperplane}
		S(K_1,\ldots,K_{n-1};{}\cdot{})
		= V^{u^\perp}\!(K_1,\ldots,K_{n-1})(\delta_u + \delta_{-u}).
	\end{equation}
	
	Next, recall that mixed area measures are \emph{locally determined}. Indeed, $S(K_1, \dots, K_{n-1}; \beta)$ depends only on $\tau(K_1, \beta), \dots, \tau(K_{n-1}, \beta)$, where the \emph{reverse spherical image} $\tau(K, \beta)$ of $K \in \R^n$ at a Borel set $\beta \subseteq \S^{n-1}$ is defined as
	\begin{equation*}
		\tau(K,\beta)
		= \bigcup_{u\in\beta} F(K,u),
	\end{equation*}
	with $F(K,u)=\{x\in\R^n: h_K(u)=\pair{x}{u}\}$ the \emph{face} of $K$ in direction $u \in \S^{n-1}$.
	\begin{lem}[{\cite{Schneider2014}*{p.~215}}]\label{lem:MixedAreaMeas_locally_det}
		Let $K_1,K_1',\ldots,K_{n-1},K_{n-1}'\in\K(\R^n)$ and $\beta\subseteq\S^{n-1}$ be a Borel set such that $\tau(K_j,\beta)=\tau(K_j',\beta)$ for all $j\in\{1,\ldots,n-1\}$. Then
		\begin{equation*}
			S(K_1,\ldots,K_{n-1};\beta)
			= S(K_1',\ldots,K_{n-1}';\beta).
		\end{equation*}
	\end{lem}

	In order to deal with the uniqueness of the solution $K$ to the Christoffel problem involving the disk, we recall
	\emph{Minkowski's quadratic inequality}, a classical result from the Brunn--Minkowski theory. It states that for all convex bodies $K,L,C\in\K(\R^n)$,
	\begin{equation}\label{eq:Minkowskis_quadr_ineq}
		V(K,L,C^{[n-2]})^2 \geq V(K,K,C^{[n-2]})V(L,L,C^{[n-2]}).
	\end{equation}
	This inequality was first established by Minkowski~\cite{Minkowski1903} in three dimensions and later extended by Bonnesen and Fenchel~\cite{Bonnesen1987} to higher dimensions. Its equality conditions were settled only recently in a landmark paper by Shenfeld and van Handel~\cite{Shenfeld2022}.
	
	For our purposes, we only require the special case where $C$ is lower-dimensional (see~\cite{Shenfeld2022}*{Theorem~2.3}).	
	Here, a vector $u\in\R^{n-1}\setminus\{0\}$ is called a \emph{$1$-extremal} normal direction of $C$ if there do not exist linearly independent normal vectors $u_1,u_2,u_3\in\R^{n-1}\setminus\{0\}$ at a boundary point of $C$ such that $u=u_1+u_2+u_3$. 
	
	\begin{thm}[{\cite{Shenfeld2022}}]\label{thm:Minkowskis_quadr_eq_cases}
		Let $K,L,C\in\K(\R^n)$ be such that $C-C\subseteq e_n^\perp$ and $V(K,L,C^{[n-2]})>0$. Then equality holds in \cref{eq:Minkowskis_quadr_ineq}		
		if and only if $\tilde{L}:=\frac{V(K,L,C^{[n-2]})}{V(L,L,C^{[n-2]})}L$ satisfies that $K+F(\tilde{L},e_n)$ and $\tilde{L}+F(K,e_n)$ have the same supporting hyperplanes in all $1$-extremal normal directions of $C$.
	\end{thm}
	
	Observe now that at every point $x\in\DD$, the corresponding normal cone $N(\DD,x)$ is either one- or two-dimensional. This implies that every direction $u\in\S^{n-1}$ is a $1$-extremal normal direction of $\DD$, so \cref{thm:Minkowskis_quadr_eq_cases} specializes for $C = \DD$ to the following statement.
	
	\begin{cor}\label{cor:MinkQuadEqDiskCases}
		Let $K,L\in\K(\R^n)$ be such that $V(K,L,\DD^{[n-2]})>0$. Then
		\begin{align*}
			V(K,L,\DD^{[n-2]})^2 = V(K,K,\DD^{[n-2]})V(L,L,\DD^{[n-2]})
		\end{align*}
		if and only if $K+F(\tilde{L},e_n) = \tilde{L}+F(K,e_n)$, where $\tilde{L}:=\frac{V(K,L,\DD^{[n-2]})}{V(L,L,\DD^{[n-2]})}L$.
	\end{cor}

	\subsection{Disintegration of measures}\label{sec:disint}

	Next, we want to show that under the given assumptions on the pushforward of the measure $\mu$ in \cref{mthm:Christoffel_disk_gnrl}, there exists a disintegration \cref{eq:disintIntro}. To this end, we recall the disintegration theorem from measure theory (see, e.g., \cite{Ambrosio2008}*{Theorem~5.3.1}).
	
	\begin{thm}[{\cite{Ambrosio2008}}]\label{thm:disinitegration}
		Let $X$, $Y$ be two Polish spaces, $\pi:X\to Y$ be a Borel map, $\mu$ be a non-negative Borel measure on $X$, and denote $\nu=\pi_\ast \mu$. Then there exists a $\nu$-a.e.\ uniquely determined family $(\mu_y)_{y\in Y}$ of probability measures on $X$ such that
		\begin{enumerate}[label=\upshape(\roman*)]
			\item for every Borel set $\beta\subseteq X$, the map $y\mapsto \mu_y(\beta)$ is Borel,
			\item for $\nu$-a.e.\ $y\in Y$, the probability measure $\mu_y$ is concentrated on $\pi^{-1}(y)$, and
			\item for every bounded Borel function $f:X\to \R$,
			\begin{equation*}
				\int_{X} f(x)\, \mu(dx)
				= \int_Y \int_{\pi^{-1}(y)} f(x)\, \mu_y(dx)\, \nu(dy).
			\end{equation*}
		\end{enumerate}
	\end{thm}
	
	This family $(\mu_y)_{y\in Y}$ is called \emph{disintegration} of $\mu$.	
	We apply this theorem in the instance where $X=\S^{n-1}\setminus \{\pm e_n\}$, $Y=\Gr_2(\R^n,e_n)$, and $\pi:\S^{n-1}\setminus \{\pm e_n\}\to \Gr_2(\R^n,e_n):u\mapsto e_n\vee u$ assigns to the point $u$ the unique $2$-plane containing both $e_n$ and $u$. Moreover, we endow the Grassmann manifold $\Gr_2(\R^n,e_n)$ with the unique $\OO(n-1)$ invariant probability measure.
	
	\begin{cor}\label{cor:disintegration_variant}
		Let $\mu$ be a non-negative Borel measure on $\S^{n-1}\setminus\{\pm e_n\}$ and suppose that $\pi_\ast \mu$ is absolutely continuous with a continuous density. Then there exists an a.e.\ uniquely defined family $(\mu_E)_{E \in \Gr_2(\R^n, e_n)}$ of non-negative Borel measures on $\S^{n-1}$ such that
		\begin{enumerate}[label=\upshape(\roman*)]
			\item for every Borel set $\beta\subseteq \S^{n-1}$, the map $E\mapsto \mu_E(\beta)$ is Borel,
			\item for a.e.\ $E\in\Gr_2(\R^n,e_n)$, the measure $\mu_E$ is concentrated on $\S^1(E)\setminus\{\pm e_n\}$, and
			\item for every bounded Borel function $f:\S^{n-1}\to \R$,
			\begin{align}\label{eq:disintMeasGr2}
				\int_{\S^{n-1}\setminus \{\pm e_n\}} f(u) \, \mu(du) = \int_{\Gr_2(\R^n, e_n)} \int_{\S^1(E)\setminus\{\pm e_n\}} f(v) \, \mu_E(dv)\, dE.
			\end{align}
		\end{enumerate}
	\end{cor}
	\begin{proof}		
		First, we let $\nu=\pi_\ast \mu$ and denote by $(\tilde{\mu}_E)_{E\in\Gr_2(\R^n,e_n)}$ the disintegration of $\mu$ according to \cref{thm:disinitegration}.		
		By our assumption on $\mu$, we have $\nu(dE)=\rho(E)dE$ with some density $\rho\in C(\Gr_2(\R^n,e_n))$. Due to the continuity of $\rho$, we may now define $\mu_E := \rho(E) \tilde{\mu}_E$ for $E\in\Gr_2(\R^n,e_n)$.
		It is then easy to see that the family $(\mu_E)$ fulfills the assertion.
	\end{proof}

	\subsection{Proof of \cref{mthm:Christoffel_disk_gnrl}}	
    
	Now we are ready to prove the main result of this article, \cref{mthm:Christoffel_disk_gnrl}. In the proof below, we readily identify $\mu$ with its restriction to $\S^{n-1}\setminus\{\pm e_n\}$, and denote by $(\mu_E)$ its disintegration according to \cref{cor:disintegration_variant}.

	\begin{proof}[Proof of \cref{mthm:Christoffel_disk_gnrl}]
		To show that the conditions \cref{it:ChristDiskAbsCont} to \cref{it:ChristDiskIntCond} are necessary, suppose that there exists a convex body $K\in\K(\R^n)$ such that $\mu=S_1(K,\DD,{}\cdot{})$. We identify $\mu$ with its restriction to $\S^{n-1}\setminus\{\pm e_n\}$,		
		take some function $\tilde{f}\in C(\Gr_2(\R^n,e_n))$ and let $f:=\tilde{f}\circ \pi$, that is, $f(u)=\tilde{f}(e_n\vee u)$. Then the Cauchy--Kubota type formula \cref{eq:zonal_Kubota_polarized_i=1}, combined with \cref{eq:MixedVol_Integral}, gives
		\begin{align*}
			&\int_{\S^{n-1}}  f(u)\, S_{1}(K,\DD,du)
			= \frac{\kappa_{n-1}}{2}  \int_{\Gr_{2}(\R^n,e_n)} \tilde{f}(E) S^E_{1}(K|E,\S^1(E))  \, dE \\
			&\qquad = \kappa_{n-1} \int_{\Gr_{2}(\R^n,e_n)} \tilde{f}(E) V_1(K|E) \, dE,
		\end{align*}
		and thus, $\pi_\ast \mu = \kappa_{n-1}V_1(K|E)dE$. Since the expression $V_1(K|E)$ is continuous in $E\in\Gr_2(\R^n,e_n)$, the push-forward measure $\pi_\ast \mu$ is absolutely continuous with a continuous density, showing \cref{it:ChristDiskAbsCont}. Moreover, due to \cref{eq:zonal_Kubota_polarized_i=1} and the uniqueness in \cref{cor:disintegration_variant}, its disintegration $(\mu_E)_E$ is given by $\mu_E=\frac{\kappa_{n-1}}{2}S_1^E(K|E,{}\cdot{})$. Hence, for a.e.\ $E\in\Gr_2(\R^n,e_n)$, the measure $\mu_E$ (as a measure on $\S^1(E)$) is centered, showing \cref{it:ChristDiskCent}.		
		To verify the final condition \cref{it:ChristDiskIntCond}, we let $h$ be the support function of $K$ and combine \cref{eq:Steiner_point} and \cref{thm:Berg_Christoffel} in dimension $2$ to obtain
		\begin{align*}
			&h(u) - \frac{1}{\pi}\int_{\S^1(e_n\vee u)} h(v)\pair{u}{v}\, dv
			= h_{K|(e_n\vee u) - s(K|(e_n\vee u))}(u) \\
			&\qquad = \frac{2}{\pi\kappa_{n-1}} \int_{\S^1(e_n\vee u)}\sqrt{1-\pair{u}{v}^2}(\pi - \arccos\pair{u}{v}) \,\mu_{e_n\vee u}(dv).
		\end{align*}
		
		Conversely, suppose that $\mu$ is a non-negative finite Borel measure on $\S^{n-1}$ satisfying the conditions \cref{it:ChristDiskAbsCont} to \cref{it:ChristDiskIntCond}, with $(\mu_E)_E$ denoting its disintegration according to \cref{cor:disintegration_variant}. Moreover, take $K\in\K(\R^n)$ to be the convex body that has $h$ as its support function. For a.e.\ subspace $E\in \Gr_2(\R^n,e_n)$, the measure $\mu_E$ is centered, so by \cref{thm:Berg_Christoffel}, there exists a convex body $K_E\in\K(E)$ such that $\mu_E=S_1^{E}(K_E,{}\cdot{})$. Combining \cref{eq:Steiner_point} and \cref{eq:Berg_Christoffel}, we obtain that
		\begin{equation*}
			h_{K|E-s(K|E)}(u)
			= \int_{\S^1(E)}\sqrt{1-\pair{u}{v}^2}(\pi - \arccos\pair{u}{v}) \,\mu_{E}(dv)
			= \pi h_{K_E-s(K_E)}(u)
		\end{equation*}
		for every $u\in\S^1(E)$. This shows that $K|E-s(K|E)= \pi (K_E - s(K_E))$ and thus, $\mu_E=S_1^E(K_E,{}\cdot{})=\frac{1}{\pi}S_1^E(K|E,{}\cdot{})$. As $\mu_E$ has no mass at $\pm e_n$, so has $S_i^E(K|E,{{}\cdot{}})$. Therefore, combining \cref{eq:zonal_Kubota_polarized_i=1} and \cref{eq:disintMeasGr2}, we have that for every bounded Borel function $f:\S^{n-1}\to\R$,
		\begin{align*}
			&\int_{\S^{n-1}} f(u) \, \mu(du)
			= \int_{\Gr_2(\R^n, e_n)} \int_{\S^1(E)\setminus\{\pm e_n\}} f(v) \, \mu_E(dv)\, dE \\
			&\qquad =  \frac{1}{\pi}\int_{\Gr_{2}(\R^n,e_n)} \!\int_{\S^{1}(E)\setminus\{\pm e_n\}} \! f(u)\, S^E_{1}(K|E,du) \, dE
			= \frac{2}{\kappa_{n-1}\pi} \int_{\S^{n-1}}  f(u)\, S_{1}(K,\DD,du).
		\end{align*}
		Hence $\mu=\frac{2}{\kappa_{n-1}\pi}S_1(K,\DD,{}\cdot{})$ and rescaling $K$ yields the claim.
		
		For the uniqueness part of the statement, take two convex bodies $K,L\in\K(\R^n)$ such that $S_1(K,\DD,{}\cdot{})=S_1(L,\DD;{}\cdot{})$, and also, $S_1(K,\DD,\{\pm e_n\})=0$. By integrating the respective support function of $K$ and $L$ against these mixed area measures, by \cref{eq:MixedVol_Integral}, we have that $V(K,K,\DD^{[n-2]}) = V(K,L,\DD^{[n-2]})=V(L,L,\DD^{[n-2]})$.
		
		Thus, equality is attained in Minkowski's quadratic inequality \cref{eq:Minkowskis_quadr_ineq} for $C=\DD$, so according to \cref{cor:MinkQuadEqDiskCases}, the bodies $K+F(L,e_n)$ and $L+F(K,e_n)$ are equal. Noting, however, that due to \cref{lem:MixedAreaMeas_locally_det} and identities \cref{eq:MixedVol_IntrinsicVol} and \cref{eq:MixedAreaMeas_hyperplane},
		\begin{equation*}
			S_1(K,\DD;\{\pm e_n\})
			= S_1(F(K,\pm e_n),\DD;\{\pm e_n\})
			= \frac{1}{n} V_1(F(K,\pm e_n)),
		\end{equation*}
		the condition $\mu(\pm e_n) = 0$ is equivalent to the faces $F(K,e_n)$ and $F(K,-e_n)$ being only singletons. Hence, we conclude that $K$ and $L$ are translates of each other.
	\end{proof}
	
	Under the addition that $\mu$ is even, which corresponds to $K$ being centrally symmetric, the classification simplifies considerably. This is the content of the theorem below, which we obtain from an easy modification of the proof of \cref{mthm:Christoffel_disk_gnrl}.
	
	\begin{thm}\label{thm:Christoffel_disk_gnrl_even}
		Let $\mu$ be a non-negative even finite Borel measure on $\S^{n-1}$ and $\mu(\{\pm e_n\})=0$. Then there exists a convex body $K\in\K(\R^n)$ such that $\mu=S_1(K,\DD,{}\cdot{})$ if and only if $\pi_\ast \mu$ is absolutely continuous with a continuous density and
		\begin{equation}\label{eq:Christoffel_disk_gnrl_even}
			h(u)
			= \int_{\S^1(e_n\vee u)}\sqrt{1-\pair{u}{v}^2} \,\mu_{e_n\vee u}(dv),			
		\end{equation}
		defined for a.e.\ $u\in\S^{n-1}$, is a support function.
		
		Moreover, $K$ is unique up to a translation.
	\end{thm}
	\begin{proof}
		Suppose first that there exists a convex body $K\in\K(\R^n)$ such that $\mu=S_1(K,\DD,{}\cdot{})$.
		By the evenness of $\mu$ and compatibility of mixed area measures with linear isometries, $\mu=S_1(-K,\DD,{}\cdot{})$. By the uniqueness statement in \cref{mthm:Christoffel_disk_gnrl}, there exists some $x\in\R^n$ such that $K=x-K$, that is, $K$ is symmetric about $x$.
		
		Due to \cref{mthm:Christoffel_disk_gnrl}, the push-forward measure $\pi_\ast \mu$ is absolutely continuous with a continuous density, and from the proof of \cref{mthm:Christoffel_disk_gnrl}, the disintegration of $\mu$ is necessarily given by $\mu_E=\frac{\kappa_{n-1}}{2}S_1^E(K|E,{}\cdot{})$ for a.e.\ $E\in\Gr_2(\R^n,e_n)$. In particular, $\mu_E$ (as a measure on $\S^1(E)$) is even for a.e.\ $E\in\Gr_2(\R^n,e_n)$.
		Moreover, from the proof of \cref{mthm:Christoffel_disk_gnrl}, the support function $h$ of a dilated copy of $K$ satisfies identity \cref{eq:Christoffel_disk_gnrl}. Due to the evenness of $h$ and $\mu_E$, this identity then simplifies to \cref{eq:Christoffel_disk_gnrl_even}.
		
		Conversely, suppose now that $\mu$ is a non-negative finite Borel measure on $\S^{n-1}$ satisfying the conditions stated in the theorem, with $(\mu_E)_E$ denoting its disintegration according to \cref{cor:disintegration_variant}. The evenness of $\mu$ and a.e.\ uniqueness of its disintegration assert that $\mu_E$ (as a measure on $\S^1(E)$) is for a.e.\ $E\in\Gr_2(\R^n,e_n)$ even, and thus, centered.
		
		Consequently, the right hand side of \cref{eq:Christoffel_disk_gnrl_even} is an even function of $u$, and thus, so is the left hand side. Since identity \cref{eq:Christoffel_disk_gnrl_even} is simply the even component of identity \cref{eq:Christoffel_disk_gnrl}, we deduce that $h$ also satisfies \cref{eq:Christoffel_disk_gnrl}. In conclusion, $\mu$ meets the requirements of \cref{mthm:Christoffel_disk_gnrl}, and thus, there exists a convex body $K\in\K(\R^n)$ such that $\mu=S_1(K,\DD,{}\cdot{})$.	
	\end{proof}

	We turn to the special case where $\mu$ is absolutely continuous with respect to the spherical Lebesgue measure, that is $\mu(du)=q(u)du$ for some $q\in L^1(\S^{n-1})$.
	We have the following decomposition of the spherical Lebesgue measure into spherical cylinder coordinates at our disposal (see, e.g., \cite{Mueller1966}*{p.~1})
	\begin{equation*}
		\int_{\S^{n-1}} q(u)\, du
		= \frac{1}{2} \int_{\S^{n-2}(e_n^\perp)} \int_{\S^1(e_n\vee u)} q(v) \abs{\pair{u}{v}}^{n-2}\, dv\, du.
	\end{equation*}	
	
	From this formula, it is immediate that $\pi_\ast \mu$ is absolutely continuous with a continuous density given by $\frac{d\pi_\ast \mu}{dE}
	= \frac{\omega_{n-1}}{2} \int_{\S^1(E)} q(v) \norm{v|E}^{n-2}\, dv$. Moreover, the disintegration of $\mu$ in the sense of \cref{eq:disintMeasGr2} is given by
	\begin{equation*}
		\mu_E(dv)
		= \frac{\omega_{n-1}}{2} q(v) (1-\pair{e_n}{v}^2)^{\frac{n-2}{2}}\, dv.
	\end{equation*}
	Hence, as a special case of \cref{mthm:Christoffel_disk_gnrl}, we obtain the following.

	\begin{cor}\label{cor:unique}
		Let $q$ be a non-negative $L^1(\S^{n-1}) $ function. Then there exists a convex body $K\in\K(\R^n)$ such that $q(u)du=S_1(K,\DD,du)$ if and only if $\int_{\S^1(e_n\vee u)}q(v)\abs{\pair{u}{v}}^{n-2}v \, dv=0$ for almost every $u\in \S^{n-2}(e_n^\perp)$, and there exists a support function $h$ such that
		\begin{align*}
			&h(u) - \frac{1}{\pi}\int_{\S^1(e_n\vee u)} h(v)\pair{u}{v}\, dv \\
			&\qquad = \int_{\S^1(e_n\vee u)}\sqrt{1-\pair{u}{v}^2}(\pi - \arccos\pair{u}{v})(1-\pair{e_n}{v}^2)^{\frac{n-2}{2}} q(v)\, dv
		\end{align*}
		for a.e.\ $u\in\S^{n-1}$.
		
		Moreover, $K$ is unique up to a translation.
	\end{cor}
	
	Under the additional assumption of evenness, \cref{thm:Christoffel_disk_gnrl_even} yields the following.
	
	\begin{cor}
		Let $q$ be a non-negative even $L^1(\S^{n-1}) $ function. Then there exists a convex body $K\in\K(\R^n)$ such that $q(u)du=S_1(K,\DD,du)$ if and only if
		\begin{equation*}
			h(u)
			= \int_{\S^1(e_n\vee u)}\sqrt{1-\pair{u}{v}^2} (1-\pair{e_n}{v}^2)^{\frac{n-2}{2}} q(v)\, dv,
		\end{equation*}
		defined for a.e.\ $u\in\S^{n-1}$, is a support function.
		
		Moreover, $K$ is unique up to a translation.
	\end{cor}


	\subsection{Discussion}\label{sec:discussion1} 
	
	
	In \cite{Brauner2025a}, the authors considered the mixed Christoffel problem for the disk area measure in the greater context of the Christoffel--Minkowski problem, but restricting the attention to the setting where all bodies are bodies of revolution with respect to a fixed axis (say $e_n$). In this setting, the following weaker version of \cref{mthm:Christoffel_disk_gnrl} was proved. Here, a measure (or a function) on $\S^{n-1}$ is called zonal, if it is invariant under rotations from $\SO(n-1)$, the stabilizer of $e_n$ inside $\SO(n)$.
	
	\begin{thm}[\cite{Brauner2025a}*{Thm.~B}]\label{thm:ZonalChristDisk}
		Let $\mu$ be a non-negative, zonal Borel measure on $\S^{n-1}$. Then there exists a body of revolution $K \in \K(\R^n)$ that is not a segment with $\mu = S_i(K, \DD; {{}\cdot{}})$ if and only if $\mu$ is centered and not concentrated on $\S^{n-1}\cap e_n^\perp$.
	\end{thm}
	
	It is remarkable that under additional symmetries, condition~\ref{it:ChristDiskIntCond} of \cref{mthm:Christoffel_disk_gnrl}, that is, the convexity of the function in \cref{eq:Christoffel_disk_gnrl}, disappears. We will give here a short argument, why this is the case.
	
	\begin{lem}\label{lem:condZonCase}
		Suppose that $\mu$ is a non-negative, zonal, centered Borel measure on $\S^{n-1}$ with $\mu(\{\pm e_n\}) = 0$, which is not concentrated on $\S^{n-2}(e_n^\perp)$. Then there exists a support function $h$ of a convex body such that
		\begin{equation}\label{eq:lemCondZonCase}
			h(u)- \frac{1}{\pi}\int_{\S^1(e_n\vee u)} h(v)\pair{u}{v}\, dv = \int_{\S^1(e_n\vee u)}\sqrt{1-\pair{u}{v}^2}(\pi - \arccos\pair{u}{v}) \,\mu_{e_n\vee u}(dv),			
		\end{equation}
		for a.e. $u \in \S^{n-1}$, that is, condition~\ref{it:ChristDiskIntCond} of \cref{mthm:Christoffel_disk_gnrl} is satisfied.
	\end{lem}
	\begin{proof}
		First note that if $\mu$ is zonal, then $\pi_\ast\mu$ is a constant multiple of the Haar measure on $\Gr_{2}(\R^n, e_n)$, in particular, absolutely continuous. Consequently, by \cref{cor:disintegration_variant}, the disintegration into measures $\mu_{e_n\vee u}$ exists and the integral in \cref{eq:lemCondZonCase} is well-defined. Moreover, the invariance implies that $\mu_{e_n\vee \tau u} = \tau_\ast \mu_{e_n\vee u}$ for all $\tau \in \SO(n-1)$. Hence, the integral in \cref{eq:lemCondZonCase} defines a zonal function on $\S^{n-1}$.
		
		Next, note that a zonal function on $\S^{n-1}$ is a support function, if and only if it is a support function on any subsphere $\S^1(e_n \vee u)$, $u\in \S^{n-2}(e_n^\perp)$. Consequently, we need to show that
		\begin{align*}
			w \mapsto \int_{\S^1(e_n\vee u)}\sqrt{1-\pair{w}{v}^2}(\pi - \arccos\pair{w}{v}) \,\mu_{e_n\vee u}(dv), \quad w \in \S^1(e_n \vee u),
		\end{align*}
		defines a support function for some fixed $u \in \S^{n-2}(e_n^\perp)$. By \cref{thm:Berg_Christoffel}, this is the case if and only if there exists a convex body $K \in \K(e_n \vee u)$ such that $\mu_{e_n\vee u} = S_1^{e_n \vee u}(K, {{}\cdot{}})$, that is, $\mu_{e_n\vee u}$ satisfies the condition of the classical Minkowski problem in $e_n \vee u$. This follows, however, from the fact that with $\mu$ also $\mu_{e_n\vee u}$ is centered and not concentrated on $\S^{n-2}(e_n^\perp) \cap \S^1(e_n \vee u)$.
	\end{proof}
	
	
	Motivated by \cref{thm:ZonalChristDisk} and \cref{lem:condZonCase}, it is a natural question to ask how strong condition~\ref{it:ChristDiskIntCond} of \cref{mthm:Christoffel_disk_gnrl} actually is. To illustrate this, we give an example of a measure which does not satisfy condition~\ref{it:ChristDiskIntCond}.
	
	\begin{exl}\label{ex:MeasNotDiskAreaMeas}
		We will define the measure by its disintegration as in \cref{eq:disintMeasGr2}. To this end, fix $u \in \S^{n-2}(e_n^\perp)$, let $0 < \varepsilon < 1$ and define
		\begin{align*}
			\Omega = \{E \in \Gr_{2}(\R^n, e_n): \, E \cap \mathrm{Cap}(u,1-\varepsilon) \neq  \emptyset \},
		\end{align*}
		where $\mathrm{Cap}(u,t) = \{ w\in\S^{n-1} : \pair{u}{w} \geq t \}$, $t \in [0,1]$. Then we set
		\begin{align*}
			\mu(\omega) = \int_{\Omega} \int_{\S^1(E)\setminus\{\pm e_n\}} \indf_\omega(v) \rho(E)\, dv\, dE,
		\end{align*}
		for all Borel sets $\omega \subset \S^{n-1}$, where $\rho \in C(\Gr_{2}(\R^n, e_n))$ is positive on $\mathrm{int}(\Omega)$ and zero on the complement of $\Omega$. Hence, $\mu_{E} = \rho(E) dv$ for $E \in \Omega$ and zero otherwise, and, clearly, $\mu$ satisfies conditions~\ref{it:ChristDiskAbsCont} and \ref{it:ChristDiskCent} of \cref{mthm:Christoffel_disk_gnrl}.
		
		Next, note that, since the integrand is positive almost everywhere,
		\begin{align}\label{eq:exNotSatPos}
			\int_{\S^1(e_n\vee u)}\sqrt{1-\pair{u}{v}^2}(\pi - \arccos\pair{u}{v}) \,\mu_{e_n\vee u}(dv) > 0,
		\end{align}
		but we have
		\begin{align}\label{eq:exNotSatZero}
			\int_{\S^1(e_n\vee w)}\sqrt{1-\pair{w}{v}^2}(\pi - \arccos\pair{w}{v}) \,\mu_{e_n\vee w}(dv) = 0
		\end{align}
		for all $w \in \S^{n-2}(e_n^\perp)$ with $|\pair{w}{u}| < 1-\varepsilon$. We conclude that \ref{it:ChristDiskIntCond} cannot be satisfied. Indeed, if \ref{it:ChristDiskIntCond} holds, that is, \cref{eq:Christoffel_disk_gnrl} defines a support function of a body $K \in \K(\R^n)$, then \cref{eq:exNotSatZero} implies that $K \subset \mathrm{span}\{e_n\}$. This contradicts \cref{eq:exNotSatPos}.
	\end{exl}
	
	Let us point out that Example~\ref{ex:MeasNotDiskAreaMeas} can easily be generalized to yield a full range of examples.


    \section{A convexity criterion}\label{sec:convCritC2p}

    In the following, we interpret the mixed Christoffel problem involving the disk as a differential equation, imposing stronger regularity on the objective body~$K$. Employing the classical characterization of convexity in terms of definiteness of the Hessian, we find a criterion to determine when the solution of said differential equation is indeed convex.
    
    The first step is to employ polar coordinates on the unit sphere in order to relate the problem to a family of ordinary differential equations. We will employ some classical tools from Riemannian geometry; as a general reference we recommend the monograph by O'Neill~\cite{ONeill1983}.

    \subsection{Polar coordinates on the unit sphere}

    In order to compute the spherical Hessian in polar coordinates, it will be convenient to regard the unit sphere $\mathbb{S}^{n-1}$, endowed with its standard round metric $g_{\mathbb{S}^{n-1}}$, as a warped product $\S^{n-1} = [0,\pi]\times_{\sin\theta}\S^{n-2}$, that is, 
    $$g_{\S^{n-1}} = d\theta^2 + \sin^2\theta \,g_{\S^{n-2}},
    $$
    where $g_{\S^{n-2}}$ denotes the standard round metric on $\S^{n-2}$ (see, e.g., \cite{ONeill1983}*{p.~204}).
    Given this representation of the metric tensor, the connection formulas are given as follows: Let $\partial_\theta$ denote the base direction and $X,Y$ be tangent to $\S^{n-2}$. Then
    \begin{align*}
    	\nabla_{\partial_\theta}\partial_\theta &= 0,\\
    	\nabla_X \partial_\theta &= \nabla_{\partial_\theta} X = \cot\theta \,X,\\
    	\nabla_X Y &= \nabla^{\S^{n-2}}_X Y - \sin\theta\cos\theta\, g_{\S^{n-2}}(X,Y)\,\partial_\theta,
    \end{align*}      
    where $\nabla^{\S^{n-2}}$ denotes the Levi--Civita connection on the $(n-2)$-dimensional sphere $\S^{n-2}$. Next, we recall that the Hessian of a function $h\in C^2(\S^{n-1})$ is defined as the covariant $2$-tensor field
    $
    \nabla^2_{\S^{n-1}}f(X,Y)
    = X Y h - \nabla_X Y h
    $
    for all smooth vector fields $X$, $Y$ on $\S^{n-1}$. Thus, the components of the Hessian on $\S^{n-1}$ are given by 
    \begin{align*}
    	\nabla_{\S^{n-1}}^2 h(\partial_\theta,\partial_\theta)
    	&=  \partial_\theta^2 h,\\
    	\nabla_{\S^{n-1}}^2 h(\partial_\theta,X)
    	&= X  \partial_\theta h - \cot\theta\,X h,\\
    	\nabla_{\S^{n-1}}^2 h(X,Y)
    	&= \nabla_{\S^{n-2}}^2 h(X,Y)
    	+ \sin\theta\cos\theta\,  \partial_\theta h\, g_{\S^{n-2}}(X,Y).
    \end{align*}
    Denoting $e^\theta = d\theta$, we can thus express the spherical Hessian as follows
    \begin{align*}
    	\nabla^2_{\S^{n-1}} h
        = \partial_\theta^2 h \,e^\theta \otimes e^\theta + \big( \nabla^{\S^{n-2}} (\partial_\theta h) - \cot\theta\, \nabla^{\S^{n-2}} h \big) \odot e^\theta 
    	    + \nabla^2_{\S^{n-2}} h + \,\sin\theta\cos\theta \, \partial_\theta h \, g_{\S^{n-2}},
    \end{align*} 
    where $\odot$ denotes the symmetric product.
    
    Next, we denote by $D^2h$ the Hessian of the one-homogeneous extension of $h$ to $\R^n\setminus\{o\}$. It is related to the spherical Hessian via the identity $D^2 h =\nabla_{\S^{n-1}}^2 h + h g_{\S^{n-1}}$ (see, e.g., \cite{Schneider2014}*{Sec.~2.5}). Hence, its representation in polar coordinates on $\S^{n-1}$ is given by
    \begin{align}\label{D2polar}
    \begin{split}
 		D^2h = (\partial_\theta^2 h + h)\,e^\theta \otimes e^\theta + \big( \nabla^{\S^{n-2}} (\partial_\theta h )- \cot\theta\, \nabla^{\S^{n-2}} h \big) \odot e^\theta\\
         + \nabla^2_{\S^{n-2}} h + \,(\sin\theta\cos\theta \, \partial_\theta h + \sin^2\theta h)\, g_{\S^{n-2}}.
         \end{split}
    \end{align}
    For the subsequent computations, it is convenient to work with an orthonormal frame with respect to the metric tensor on the sphere $\S^{n-1}$. So for a given orthonormal frame $\{\hat{e}_1,\hat{e}_2,\dots,\hat{e}_{n-2}\}$ in $\S^{n-2}$, we set
    $$
    e_\theta = \partial_\theta \qquad\text{and}\qquad e_{i} = \frac{1}{\sin\theta}\hat{e}_{i},
    $$
    so that $\{e_\theta,e_1,\dots,e_{n-2}\}$ is an orthonormal frame in $\S^{n-1}$. Thus, the matrix representation of $D^2h$ in this frame $\{e_\theta,e_1,\dots,e_{n-2}\}$ is given by 
    \begin{equation}\label{eq:Hessian_polar_coords}
        \begin{bmatrix}
    		\partial_\theta^2 h + h
    		&
    		\dfrac{1}{\sin\theta}\,\nabla^{\S^{n-2}}(\partial_\theta h)
    		- \dfrac{\cot\theta}{\sin\theta}\,\nabla^{\S^{n-2}}h
    		\\[0.5em]
    		\left(
    		\dfrac{1}{\sin\theta}\,\nabla^{\S^{n-2}}(\partial_\theta h)
    		- \dfrac{\cot\theta}{\sin\theta}\,\nabla^{\S^{n-2}}h
    		\right)^{\!\top}
    		&
    		\dfrac{1}{\sin^2\!\theta}\,\nabla_{\S^{n-2}}^2h
    		+ (\cot\theta\,\partial_\theta h + h)\,I_{n-2}
    	\end{bmatrix}.
    \end{equation}

    Let $U \subseteq \S^{n-1}$ be open and $f\in C(\S^{n-1})\cap C^2(U)$. By abuse of notation, we denote by $D^2f(u)$ the Hessian of the one-homogeneous extension of $f$ at a point $u\in U$ viewed as an endomorphism of the subspace $u^\perp$. Moreover, we denote by $\mathsf{D}$ the symmetric polarization of the determinant, called \emph{mixed discriminant}:
    Given a family of endomorphisms $A_1\ldots,A_{n-1}$ of $(n-1)$-dimensional Euclidean space, the mixed discriminant $\mathsf{D}(A_1, \ldots, A_{n-1})$ is defined implicitly by
    \begin{equation*}
        \det(\lambda_1 A_1 + \cdots \lambda_{n-1}A_{n-1})
        = \sum_{j_1,\ldots,j_{n-1}=1}^{n-1} \lambda_{j_1}\cdots\lambda_{j_{n-1}} \mathsf{D}(A_{j_1},\ldots,A_{j_{n-1}}), \quad \lambda_1,\ldots,\lambda_{n-1}\in\R.
    \end{equation*}
    Given a family of convex bodies $K_1,\ldots,K_{n-1}\in\K(\R^n)$ such that their respective support functions are $C^2$~functions on some open subset $U\subseteq \S^{n-1}$, we have
    \begin{equation*}
        S(K_1,\ldots,K_{n-1};du)
        = \mathsf{D}(D^2h_{K_1}(u),\ldots,D^2h_{K_{n-1}}(u))du
        \qquad\text{on }U.
    \end{equation*} 
    Moreover, by a change of variables, the spherical Lebesgue measure decomposes as follows:
    \begin{equation*}
        du = (\sin\theta)^{n-2}d\theta\, dw.
    \end{equation*}
    In particular, the mixed are measure can be written as
    \begin{equation}\label{eq:mixed_area_meas_polar_coords}
        S(K_1,\dots,K_{n-1},du)
        = \mathsf{D}\big(D^2 h_{K_1}(\theta,w),\dots,D^2 h_{K_{n-1}}(\theta,w)\big)\,
        (\sin\theta)^{n-2}\,d\theta\,dw
        \qquad\text{on }U.
    \end{equation}
    We apply this to express the mixed area measure involving the disk in polar coordinates $(\theta,w)$.
    


    \begin{lem}\label{polar_density}
    Let $K \in \mathcal{K}^n$ be a convex body of class $C^2_+$. The mixed area measure $S_1(K,\mathbb{D},\,\cdot\,)$ has density in polar coordinates given by
    \begin{equation}\label{eq:polar_density}
    S_1(K,\mathbb{D};du)
    = \frac{1}{n-1}\big(\partial_\theta^2 h_K(\theta,w) + h_K(\theta,w)\big) d\theta\, dw.
    \end{equation}
    \end{lem}
\begin{proof}
First note that since $K$ is of class $C^2_+$, the face $F(K,\pm e_n)$ in direction $\pm e_n$ consists of single points $x_\pm$. Hence, $S_1(K,\DD;\{\varepsilon e_n\}) = S_1(F(K,\varepsilon e_n),\DD;\{\varepsilon e_n\}) = S_1(\{x_{\varepsilon}\},\DD;\{\varepsilon e_n\}) = 0$ for $\varepsilon\in \{\pm 1\}$, and we can work on $\S^{n-1}\setminus\{\pm e_n\}$.

Next, again since $K$ is of class $C^2_+$, by \cref{eq:mixed_area_meas_polar_coords}, the measure $S_1(K,\mathbb{D},\,\cdot\,)$ has density with respect to the spherical Lebesgue measure $du$ on $\mathbb{S}^{n-1}\setminus\{e_n\}$ given by
    $$
    \mathsf{D}(D^2 h_K(u), D^2 h_{\mathbb{D}}       (u)^{[n-2]})
    = \frac{1}{n-1}\operatorname{tr}\big(\operatorname{cof}      (D^2 h_{\mathbb{D}}(u))\, D^2 h_K(u)\big).
    $$
    In polar coordinates $(\theta,w)$ on $\mathbb{S}^{n-1}$, we have
    $
    h_{\mathbb{D}}(\theta,w) = \sin\theta.
    $
    Since $h_{\mathbb{D}}$ depends only on $\theta$, a direct computation (using \eqref{D2polar}) shows that
    $$
    D^2 h_{\mathbb{D}}(\theta,w) = \frac{1}{\sin\theta}\, g_{\mathbb{S}^{n-2}}.
    $$
    Thus $D^2 h_{\mathbb{D}}(u)$ has eigenvalue $\frac{1}{\sin\theta}$ in directions tangent to $\mathbb{S}^{n-2}$ and eigenvalue $0$ in the $\theta$–direction, and thus,
    $$
    \operatorname{cof}(D^2 h_{\mathbb{D}}(\theta,w))
    = \frac{1}{(\sin\theta)^{n-2}}\,
    e^\theta \otimes e^\theta.
    $$
    Hence,
    $$
    \operatorname{tr}\big(\operatorname{cof}(D^2    h_{\mathbb{D}}(\theta,w))\, D^2 h_K(\theta,w)\big)
    = \frac{1}{(\sin\theta)^{n-2}}
    \langle D^2 h_K(\theta,w)e_\theta,\, e_\theta\rangle.
    $$
    By (\ref{D2polar}), the matrix $D^2h_K$ satisfies
    $$
    \langle D^2 h_K(\theta,w)e_\theta, e_\theta\rangle
    = \partial_\theta^2 h_K(\theta,w) +        h_K(\theta,w).
    $$
    Combining the expressions above, we obtain
    $$
    dS_1(K,\mathbb{D},\,\cdot\,)
    = \frac{1}{n-1}\big(\partial_\theta^2 h_K(\theta,w) + h_K(\theta,w)\big)
    \, d\theta\, dw,
    $$
    concluding the computation.
    \end{proof}

    \subsection{Differential equations along meridians}
    
    In view of Lemma~\ref{polar_density}, our problem can be expressed as an infinite system of ordinary differential equations along the meridians: We want to find a function $h : \S^{n-1} \to \mathbb{R}$
    such that, for each fixed $w \in \mathbb{S}^{n-2}$,
     \begin{equation}\label{ODEsmeridians}
    \partial_\theta^2 h(\theta,w) + h(\theta,w)
    = (n-1)q(\theta,w),
    \qquad \theta\in(0,\pi), \quad w\in \S^{n-2}.
     \end{equation}
    subject to the Dirichlet boundary conditions
    $$
    h(0,w) = a \qquad \text{and} \qquad h(\pi,w) = b,
    $$
    where $a,b\in\mathbb{R}$ are constants independent of $w$.
    For each fixed $w\in\S^{n-1}$, this yields a second order ordinary differential equation with Dirichlet boundary conditions. We can treat this problem by simply applying some standard theory of ordinary differential equations (see, e.g., \cite{Heuser2006}), but for the convenience of the reader, we provide a short proof. 
    
    \begin{prop}\label{prop:Fredholm_alternative_ODE}
    	Let $f\in C([0,\pi])$ and $a,b\in\mathbb{R}$. The boundary value problem
    	\begin{equation}\label{eq:ODE_dirichlet}
    	\begin{cases}
    		h''(\theta) + h(\theta) = f(\theta), & \theta\in(0,\pi),\\
    		h(0) = a, \quad h(\pi) = b.
    	\end{cases}
    	\end{equation}
        admits a solution $h\in C^2([0,\pi])$ if and only if
        \begin{equation}\label{eq:ODE_dirichlet_cond}
            \int_0^\pi f(\theta)\sin\theta\,d\theta = a + b.
        \end{equation}
        In this case, $h\in C^2([0,\pi])$ is a solution to (\ref{eq:ODE_dirichlet}) if and only if there exists $c\in\R$ such that
        \begin{equation}\label{eq:ODE_dirichlet_sol}
            h(\theta)
        	= \int_{0}^{\theta} \sin(\theta-\sigma)\, f(\sigma)\, d\sigma
        	+ a \cos\theta + c \sin\theta,
    	\qquad \theta \in [0,\pi].
        \end{equation}
    \end{prop}
    \begin{proof}
        If~\cref{eq:ODE_dirichlet} admits a solution $h\in C^2([0,\pi])$, then twofold integration by parts yields
        \begin{align*}
            \int_0^\pi f(\theta)\sin\theta\,d\theta
            &= \int_0^\pi h(\theta)\sin\theta\,d\theta + \int_0^\pi h''(\theta)\sin\theta\,d\theta \\
            &= \int_0^\pi h(\theta)\sin\theta\,d\theta - \int_0^\pi h(\theta)\sin\theta\,d\theta + h(0)+ h(\pi)
            = a+b,
        \end{align*}
        and therefore, condition~\cref{eq:ODE_dirichlet_cond} is necessary.

        Conversely, suppose that $f$ satisfies this condition and let $h$ be defined by~\cref{eq:ODE_dirichlet_sol}. Clearly, $h(0)=a$ and by~\cref{eq:ODE_dirichlet_cond}, we also have $h(\pi)=b$. Applying the Leibniz integral rule twice shows that $h$ is twice differentiable and
        \begin{equation*}
            \frac{d^2}{d\theta^2} \int_{0}^{\theta} \sin(\theta-\sigma)\, f(\sigma)\, d\sigma
            = f(\theta) - \int_{0}^{\theta} \sin(\theta-\sigma)\, f(\sigma)\, d\sigma.
        \end{equation*}
        Since $f\in C([0,\pi])$, we obtain that $h\in C^2([0,\pi])$ and satisfies the differential equation $h''+h=f$. Hence $h$ is a solution to~\cref{eq:ODE_dirichlet}.
        
        Finally, if $\tilde{h} \in C^2([0,\pi])$ is another solution to the boundary problem~\cref{eq:ODE_dirichlet}, then $h_0=h-\tilde{h}\in C^2([0,\pi])$ is a solution to the homogeneous boundary problem
        \begin{equation*}
    	\begin{cases}
    		h_0''(\theta) + h_0(\theta) = 0, & \theta\in(0,\pi),\\
    		h_0(0) = 0, \quad h_0(\pi) = 0.
    	\end{cases}
    	\end{equation*}
        Thus $h_0(\theta)=c \sin\theta$ for some $c\in\R$ which shows that indeed all solutions $h\in C^2([0,\pi])$ are of the form~\cref{eq:ODE_dirichlet_sol}.
    \end{proof}

    In the same spirit as in \cref{sec:convol_rep}, the idea is now to simply glue the solutions from each meridian together to a smooth solution on the unit sphere, including the north and south pole. In this process, the integral condition~\cref{eq:ODE_dirichlet_cond} will give rise to the first of the following two conditions on the unit sphere. The second one corresponds to differentiability at north and south pole.
    \begin{equation}\tag{C}\label{eq:centeredness_condition_strong}
            \begin{gathered}
            \text{There exist $t\in\R$ and $x \in\R^{n-1}$ such that for all $w\in\S^{n-2}$,} \\
            \int_{0}^{\pi} \sin\theta f(\theta,w) \,d\theta = t \qquad \text{and} \qquad
            \int_{0}^\pi \cos\theta \, f(\theta, w) \,d\theta = \langle x, w\rangle.
            \end{gathered}
    \end{equation}
    We state the condition here so that we can just refer to it in the following and state the results more concisely. From \cref{prop:Fredholm_alternative_ODE}, we obtain the following theorem.   

 \begin{thm}\label{thm:PDE_dirichlet}
        Let $f\in C([0,\pi]\times\S^{n-2})$ be such that $f(\theta,\,\cdot\,)\in C^2(\S^{n-2})$ for all $\theta \in [0,\pi]$.
        The problem
        \begin{equation}\label{eq:PDE_dirichlet}
		      \partial_\theta^2 h(\theta,w) + h(\theta,w) = f(\theta, w), \qquad \theta\in(0,\pi), \quad \,w\in \S^{n-1}.
        \end{equation}
        admits a solution $h\in C^2(\S^{n-1})$ if and only if $f$ satisfies condition~\eqref{eq:centeredness_condition_strong}.  
        
        In this case, $h\in C^2(\S^{n-1})$ is a solution to (\ref{eq:PDE_dirichlet}) if and only if there exist $a\in\R$ and $z\in\R^{n-1}$ such that
        \begin{equation}\label{eq:PDE_dirichlet_solution}
        h(\theta,w) = \int_{0}^\theta \sin(\theta - \sigma)f(\sigma, w)d\sigma   + a\,\cos\theta + \langle z,w\rangle \,\sin\theta, \qquad \theta\in [0,\pi],\quad  w\in \S^{n-2}.
        \end{equation}
        \end{thm}
    \begin{proof}
        Suppose first that \eqref{eq:PDE_dirichlet} admits a solution $h\in C^2(\S^{n-1})$. Then, according to \cref{prop:Fredholm_alternative_ODE},
        \begin{equation*}
            \int_{0}^{\pi} \sin\theta f(\theta,w) \,d\theta = h(e_n) + h(-e_n),
            \qquad w\in\S^{n-1},
        \end{equation*}
        and the solution $h$ is of the form
        \begin{equation*}
        h(\theta,w) = \int_{0}^\theta \sin(\theta - \sigma)f(\sigma, w)d\sigma   + h(e_n) \,\cos\theta + c(w) \,\sin\theta, \qquad \theta\in [0,\pi],\quad  w\in \S^{n-1}.
        \end{equation*}
        Moreover, since $h\in C^2(\S^{n-1})$, we may consider its first-order derivatives at the north and south pole. For all $w\in\S^{n-2}$,
        \begin{align*}
            \pair{\nabla h(e_n)}{w}
            &= \left.\frac{d}{d\theta}\right|_{0}h(\cos \theta \, e_n + \sin\theta \, w)
            = \left.\frac{d}{d\theta}\right|_{0}h(\theta,w)
            = \partial_\theta h(0,w), \qquad\text{and} \\
            \pair{\nabla h(- e_n)}{w}
            &= \left.\frac{d}{d\theta}\right|_{0}h(\cos \theta \, (-e_n) + \sin\theta \, w)
            = - \left.\frac{d}{d\theta}\right|_{0} h(\pi-\theta,w)
            = \partial_\theta h(\pi,w).
        \end{align*}
        On the other hand, the integral representation of $h$ yields that for all $w\in\S^{n-2}$,
        \begin{equation*}
        \partial_{\theta} h(0,w) = c(w) \qquad\text{and}\qquad
        \partial_{\theta}h(\pi,w) = \int_{0}^{\pi} \cos\theta f(\theta, w)d\theta  - c(w).
        \end{equation*}
        Combining these computations, we obtain that necessarily,
        \begin{equation*}
        c(w) = \pair{z}{w}
        \qquad\text{and}\qquad
            \int_{0}^\pi \cos\theta \, f(\theta, w) \,d\theta = \langle x, w\rangle,\qquad w\in\S^{n-2},
        \end{equation*}
        where $z=\nabla h(e_n)$ and $x=\nabla h(e_n) + \nabla h(-e_n)$.
        
        Conversely, suppose that $f\in C(\S^{n-1})$ satisfies the conditions of the theorem.
        Take any $a\in\R$ and $z\in\R^{n-1}$, and define $h$ by the the formula \eqref{eq:PDE_dirichlet_solution}.
        Since $f\in C(\S^{n-1})$ and $f(\theta,\,\cdot\,)\in C^2(\S^{n-2})$ for all $\theta \in [0,\pi]$, it follows that $h\in C^2(\S^{n-1}\setminus \{\pm e_n\})$. Moreover, by \cref{prop:Fredholm_alternative_ODE}, we have that $h(\,\cdot\, ,w)\in C^2([0,\pi])$ for all $w\in\S^{n-2}$, and that $h$ is a solution to \eqref{eq:PDE_dirichlet}. It remains to show that $h$~is indeed a $C^2(\S^{n-1})$ function. To this end, observe that for any fixed $w\in\S^{n-2}$, the limits of $\nabla h(\theta,w)$ as $\theta \to 0$ or $\theta \to \pi$ exist in the space $\R^n$. Due to the fact that
        \begin{align*}
            \partial_\theta h(0,w)
            = \pair{z}{w}
            \qquad\text{and}\qquad
            \partial_\theta h(\pi,w)
            = \int_{0}^{\pi} \cos\theta f(\theta, w)d\theta - \pair{z}{w}
            = \pair{x-z}{w},
        \end{align*}
        the covariant derivatives of $h$ at the poles exist and are given by $\nabla h(e_n)=z$ and $\nabla h(-e_n)=x-z$. In particular, they do not depend on $w$, so $f\in C^1(\S^{n-1})$. Similarly, for any fixed $w\in\S^{n-2}$, the limits of $\nabla^2 h(\theta,w)$ as $\theta \to 0$ or $\theta \to \pi$ exist in the space $\R^{n\times n}$. Since
        \begin{align*}
            \partial_\theta^2 h(0,w)
            &= - h(0,w) + f(0,w)
            = f(e_n) - h(e_n)
            \qquad\text{and}\\
            \partial_\theta^2 h(\pi,w)
            &= - h(\pi,w) + f(\pi,w)
            = f(-e_n) - h(-e_n),
        \end{align*}
        it follows that the second-order covariant derivatives of $h$ exist at the poles are given by $\nabla^2 h(e_n)=(f(e_n)-h(e_n))I_{n-1}$ and $\nabla^2 h(-e_n)=(f(-e_n)-h(-e_n))I_{n-1}$. As before, due to the independence of~$w$, we deduce that $h\in C^2(\S^{n-1})$. This concludes the proof.
    \end{proof}

    \medskip

    We now return to the Christoffel type problem involving the disk, and recall that mixed area measures always have their centroid at the origin. The following lemma clarifies the relationship between this centeredness and condition~\cref{eq:centeredness_condition_strong}.

    \begin{lem}
        If $f\in C([0,\pi]\times \S^{n-2})$ satisfies condition~\eqref{eq:centeredness_condition_strong}, then the signed Radon measure $\mu(du)= f(\theta,w)\, d\theta\, dw$ on $\S^{n-1}$ has centroid at the origin. 
    \end{lem}
    \begin{proof}
        Direct computation shows that
        \begin{align*}
            &\int_{\S^{n-1}} u\, \mu(du)
            = \int_{\S^{n-2}} \int_{0}^{\pi} (\cos\theta\, e_n + \sin\theta\, w) f(\theta,w)\, d\theta\, dw \\
            &\qquad =   \int_{\S^{n-2}} \int_{0}^{\pi} \cos\theta\, f(\theta,w)\, d\theta\, dw \, e_n +  \int_{\S^{n-2}} \int_{0}^{\pi} \sin\theta f(\theta,w)\, d\theta\, w\, dw \\
            &\qquad =  \int_{\S^{n-2}} \pair{x}{w}\, dw \, e_n +  t \int_{\S^{n-2}} w\, dw
            = 0e_n + to
            = o,
        \end{align*}
        where we used the fact that the integral of any odd function on $\S^{n-2}$ must be zero.
    \end{proof}

    In order to determine when the solutions to \cref{eq:PDE_dirichlet} extend to a support function on $\R^n$, we want to employ the classical characterization in terms of positive definiteness of the Hessian. In the polar coordinates $(\theta,w)$, one coordinate plays a different role then the others, giving rise to the block matrix decomposition~\cref{eq:Hessian_polar_coords}. We want to exploit this decomposition in determining the definiteness, using the notion of the \emph{Schur complement}.
    
    \begin{defi}
        Consider a symmetric matrix $X\in\R^{(k+\ell)\times (k+\ell)}$ of the form    
    \begin{equation*}
        X = \begin{bmatrix}
            A   & B^{\top} \\
            B & C
        \end{bmatrix},
    \end{equation*}
    where $A\in\R^{k\times k}$, $C\in\R^{\ell\times\ell}$ are symmetric and $B\in\R^{\ell \times k}$.    
    If the block $A$ is invertible, then its \emph{Schur complement} is defined as $X/A = C - B^{\top} A^{-1} B$. 
    \end{defi}

    This notion allows to equate the positive definiteness of the matrix $X$ to the positive definiteness of the two smaller matrices $A$ and $X/A$. This is the content of the following classical and very easy theorem. 
    
    \begin{thm}\label{thm:Schur_complement}
        Let $X$ be as above and $A$ be invertible. Then $X$ is positive definite if and only if $A$ and $X/A$ are each positive definite.
    \end{thm}
    \begin{proof}
        Direct computation shows that
        \begin{equation*}
            X = \begin{bmatrix}
            I_{k}   & 0 \\
            B^{\top}A^{-1} & I_{\ell}
        \end{bmatrix}\begin{bmatrix}
            A   & 0 \\
            0 & X/A
        \end{bmatrix}\begin{bmatrix}
            I_k   & A^{-1} B \\
            0 & I_{\ell}
        \end{bmatrix},
        \end{equation*}
        and thus, $X$ is congruent to a block diagonal matrix of $A$ and $X/A$.
    \end{proof}

    We now have all ingredients in place to prove our second main result, \cref{mthm:Christoffel_disk_C2+}, which we repeat here for the reader's convenience. The integral representation of the support function of $K$ is a direct consequence of \cref{thm:PDE_dirichlet}.

    \begin{thm}\label{thm:disk_Christoffel_convcond}
        Let $q\in C([0,\pi]\times \S^{n-2})$ be strictly positive and $q(\theta,\,\cdot\,)\in C^2(\S^{n-2})$ for all $\theta \in [0,\pi]$.
        There exists some convex body $K\in\K(\R^n)$ of class $C^2_+$ such that $q(\theta,w)d\theta\, dw=S_1(K,\DD,du)$ if and only if $q$ satisfies condition~\eqref{eq:centeredness_condition_strong} and for all $(\theta,w)\in (0,\pi) \times \S^{n-2}$ and unit tangent vectors $\xi \in T_w \S^{n-2}$,
        \begin{align}\label{eq:disk_Christoffel_convcond}
        \begin{split}
            &\sin^2\!\theta\, q(\theta,w)\int_{0}^\theta[\sin(\theta - \sigma)\xi^\top\nabla^2_{\S^{n-2}}q(\sigma,w) \,\xi  +  \sin\theta\cos\sigma q(\sigma,w) ] \,d\sigma \\
            &\qquad > \bigg(\int_{0}^\theta\sin\sigma\, \xi^{\top}\nabla^{\S^{n-2}}q(\sigma,w)\,d\sigma\bigg)^{\! 2}\! .
            \end{split}
        \end{align}
        In this case, the body $K$ is unique up to translations.
    \end{thm}
    \begin{proof}
        First, note that by \cref{polar_density} and \cref{thm:PDE_dirichlet}, for a $C^2_+$ convex body $K\in\K^n$ to exist such that $q(\theta,w)d\theta\, dw=S_1(K,\DD,du)$, it is necessary that $q$ satisfies condition \cref{eq:centeredness_condition_strong}.
    
        Conversely, if $q$ satisfies condition~\eqref{eq:centeredness_condition_strong}, \cref{thm:PDE_dirichlet} asserts that the function $h$, defined as
        \begin{equation*}
        h(\theta,w) = \int_0^{\theta} \sin(\theta - \sigma)q(\sigma,w)\, d\sigma,
        \qquad \theta\in[0,\pi], \quad w\in\S^{n-1},
        \end{equation*}
        is a $ C^2(\S^{n-1})$ function that solves the differential equation~\eqref{eq:PDE_dirichlet}. Moreover, any other solution to~\eqref{eq:PDE_dirichlet} differs from $h$ by a linear function restricted to the unit sphere. Hence, due to~\eqref{eq:polar_density}, there exists some convex body $K\in\K^n$ such that $q(u)du = S_1(K,\DD,du)$ if and only if $h$ is a support function, and in that case, $\frac{1}{n-1}h$ is the support function of $K$ up to a translation.

        We recall that $h$ is the support function of a $C^2_+$ convex body if and only if the Hessian of its one-homogeneous extension, denoted as $D^2h(u)$, is positive definite at every point $u\in\S^{n-1}$ (see, e.g., \cite{Schneider2014}*{Section~1.7}). Moreover by \cref{eq:Hessian_polar_coords}, $D^2h(\theta,w)$ can be represented by the matrix
        \begin{equation*}
        X
        =   \begin{bmatrix}
            a & b^{\top} \\
            b & C
        \end{bmatrix},
    \end{equation*}
      where the blocks $a\in\R$, $b\in\R^{n-2}$, and  $C\in\R^{(n-2)\times (n-2)}$ are given by $a = \partial_\theta^2 h + h$,
    \begin{align*}
        b^\top &\,=\, \dfrac{1}{\sin\theta}\,\nabla^{\S^{n-2}}\partial_\theta h
    	- \dfrac{\cot\theta}{\sin\theta}\,\nabla^{\S^{n-2}}h,\qquad\text{and}\\
    	C &\,=\, \dfrac{1}{\sin^2\!\theta}\,\nabla_{\S^{n-2}}^2h
    	+ (\cot\theta\,\partial_\theta h + h)\,I_{n-2}.
    \end{align*}
    Since $h$ solves the differential equation~\eqref{eq:PDE_dirichlet}, we have that $a=\partial_\theta^2 h + h=q$, and thus, is strictly positive by our assumption on $q$. Hence, by \cref{thm:Schur_complement}, the matrix $X$ is positive definite if and only if the Schur complement
    \begin{equation*}
        S(\theta,w) = X/a = C - \frac{1}{a}b\otimes b,
    \end{equation*}
    understood as a self-adjoint operator on $T_w\S^{n-2}$, is positive definite.

    Hence, we need to express the Schur complement $S(\theta,w)$ in terms of the given data $q$. In order to compute the block $b$, note that
    \begin{align*}
        \nabla^{\S^{n-2}}h(\theta,w)
        &= \int_{0}^\theta \sin(\theta - \sigma)\nabla^{\S^{n-2}} q(\sigma,w)\,d\sigma,\\
        \nabla^{\S^{n-2}} \partial_\theta  h(\theta,w)
        &= \int_{0}^\theta \cos(\theta - \sigma)\nabla^{\S^{n-2}}q(\sigma,w)\,d\sigma,
    \end{align*}
    and thus,
    \begin{align*}
        b
        &= \dfrac{1}{\sin\theta}\,\nabla^{\S^{n-2}}\partial_\theta h(\theta,w)
    	- \dfrac{\cot\theta}{\sin\theta}\,\nabla^{\S^{n-2}}h(\theta,w) \\
        &= \frac{1}{\sin\theta} \int_{0}^\theta [\cos(\theta - \sigma) - \cot\theta \sin(\theta - \sigma)]\nabla^{\S^{n-2}}q(\sigma,w) d\sigma = \frac{1}{\sin^2 \theta}\int_{0}^\theta \sin\sigma\,\nabla^{\S^{n-2}}q(\sigma,w)\,d\sigma,
    \end{align*}
    where we used the trigonometric identity
    $ \cos(\theta - \sigma) - \cot\theta \sin(\theta - \sigma)
    = {\sin\sigma}/{\sin\theta}$.

    In order to compute the block $C$, note that
    \begin{align*}
         \nabla_{\S^{n-2}}^2 h(\theta,w)
         &= \int_{0}^\theta \sin(\theta - \sigma)\nabla_{\S^{n-2}}^2 q(\sigma,w)\,d\sigma,
    \end{align*}
    and thus,
    \begin{align*}
    	C
        &= \dfrac{1}{\sin^2\!\theta}\,\nabla_{\S^{n-2}}^2h(\theta,w)+  (\cot\theta\,\partial_\theta h(\theta,w) + h(\theta,w))\, I_{n-2} \\
    		&=  \dfrac{1}{\sin^2\!\theta} \int_{0}^\theta \sin(\theta - \sigma)\nabla_{\S^{n-2}}^2 q(\sigma,w) \,d\sigma +  \int_{0}^\theta [ \cot\theta\cos(\theta - \sigma) + \sin(\theta - \sigma)]q(\sigma,w) \,d\sigma\, I_{n-2} \\
    		& =  \frac{1}{\sin^2 \theta} \int_{0}^\theta \sin(\theta - \sigma)\nabla_{\S^{n-2}}^2 q(\sigma,w) \, d\sigma +   \frac{1}{\sin\theta}\int_{0}^\theta \cos\sigma q(\sigma,w) \, d\sigma  \,I_{n-2},
    \end{align*}
    where we used the trigonometric identity    
    $ \cot\theta\,\cos(\theta - \sigma) + \sin(\theta - \sigma)
    = {\cos\sigma}/{\sin\theta}$.
    
    Putting all together yields
    \begin{align*}
        \xi^{\top} S(\theta,w) \xi
        &=  \xi^{\top} C(\theta,w) \xi - \frac{1}{q(\theta,w)} (\xi^{\top} b(\theta,w))^2 \\
        &= \frac{1}{\sin^2\theta} \int_{0}^\theta \sin(\theta - \sigma)\, \xi^{\top} \nabla_{\S^{n-2}}^2  q(\sigma,w) \xi \, d\sigma + \frac{1}{\sin\theta}\int_{0}^\theta \cos\sigma q(\sigma,w) \, d\sigma\\
        &\qquad - \frac{1}{(\sin\theta)^{4}\,q(\theta,w)} \bigg(\int_{0}^\theta \sin\sigma\,\xi^{\top}\nabla^{\S^{n-2}}q(\sigma,w) \, d\sigma\bigg)^2.
    \end{align*}
    By our considerations at the beginning of the proof, $h$ is the support function of some $C^2_+$ convex body $K\in\K(\R^n)$ if and only $\xi^{\top}S(\theta,w)\xi > 0$ for all $(\theta,w)\in (0,\pi) \times \S^{n-2}$ and unit tangent vectors $\xi \in T_w \S^{n-2}$. This concludes the argument.
    \end{proof}

    \subsection{Discussion}\label{sec:discussion2}

    Consider again the case where the given density $q$ is a zonal function on $\S^{n-1}$, or equivalently, the corresponding body $K$ is a body of revolution. Then $q$ depends only on the polar angle $\theta$, and thus, $q(\theta,w)=\tilde{q}(\theta)$ for some function $\tilde{q}$. Suppose now that  $\tilde{q}\in C([0,\pi])$ is strictly positive. Condition~\cref{eq:centeredness_condition_strong} is then equivalent to  
    \begin{equation*}
        \int_0^\pi \cos\theta\, \tilde{q}(\theta) = 0.
    \end{equation*}
    Since $\tilde{q}$ and $\sin$ are strictly positive on $(0,\pi)$, the convexity condition \cref{eq:disk_Christoffel_convcond} reads as
    \begin{equation*}
        \int_0^\theta \cos\sigma\, \tilde{q}(\sigma) > 0,
        \qquad \text{for all }\theta\in (0,\pi).
    \end{equation*}
    For $\theta\in (0,\frac{\pi}{2}]$, this condition is clearly satisfied. For $\theta\in (\frac{\pi}{2},\pi)$, we may use the identity above to rewrite the integral on the left hand side as an integral over the domain $(\theta,\pi)$ with a positive integrand. Hence, the convexity condition will always be satisfied by $q$. As an immediate consequence of \cref{thm:disk_Christoffel_convcond}, we obtain the following corollary.

    \begin{cor}
        Let $q$ be a strictly positive, centered, and zonal function on $\S^{n-1}$. Then there exists a body of revolution $K\in\K(\R^n)$ of class $C^2_+$ such that $q(u)du=S_1(K,\DD;du)$.
    \end{cor}



        \section*{Acknowledgments}
    The authors would like to thank Karol\'y B\"or\"oczky for pointing out reference \cite{Li2021} to us.

	The first-named author was supported by the Deutsche Forschungsgesellschaft (DFG), project \href{https://gepris.dfg.de/gepris/projekt/520350299}{520350299}. The second-named author was supported by the Austrian Science Fund (FWF), project \href{https://doi.org/10.55776/P34446}{doi:10.55776/P34446} and project \href{https://doi.org/10.55776/ESP9378724}{doi:10.55776/ESP9378724}, ESPRIT program.
	The third-named author was supported by the Austrian Science Fund (FWF), project \href{https://doi.org/10.55776/ESP236}{doi:10.55776/ESP236}, ESPRIT program.


    \begingroup
	\let\itshape\upshape
	
	\bibliographystyle{abbrv}
	\bibliography{references}{}

\begin{bibdiv}
\begin{biblist}

\bib{Ambrosio2008}{book}{
      author={Ambrosio, Luigi},
      author={Gigli, Nicola},
      author={Savar\'e, Giuseppe},
       title={Gradient flows in metric spaces and in the space of probability
  measures},
     edition={Second},
      series={Lectures in Mathematics ETH Z\"urich},
   publisher={Birkh\"auser Verlag, Basel},
        date={2008},
        ISBN={978-3-7643-8721-1},
      review={\MR{2401600}},
}

\bib{Berg1969}{article}{
      author={Berg, Christian},
       title={Corps convexes et potentiels sph\'{e}riques},
        date={1969},
        ISSN={0023-3323},
     journal={Mat.-Fys. Medd. Danske Vid. Selsk.},
      volume={37},
      number={6},
       pages={64 pp. (1969)},
      review={\MR{254789}},
}

\bib{Bonnesen1987}{book}{
      author={Bonnesen, T.},
      author={Fenchel, W.},
      editor={Boron, L.},
      editor={Christenson, C.},
      editor={Smith, B.},
       title={Theory of convex bodies},
   publisher={BCS Associates, Moscow, ID},
        date={1987},
        ISBN={0-914351-02-8},
      review={\MR{920366}},
}

\bib{Brauner2025a}{article}{
      author={Brauner, Leo},
      author={Hofst\"atter, Georg~C.},
      author={Ortega-Moreno, Oscar},
       title={{M}ixed {C}hristoffel--{M}inkowski problems for bodies of
  revolution},
        date={2025},
      eprint={arXiv:2508.09794},
}

\bib{Brauner2024a}{article}{
      author={Brauner, Leo},
      author={Hofst\"atter, Georg~C.},
      author={Ortega-Moreno, Oscar},
       title={The {K}lain approach to zonal valuations},
        date={2026-02},
        ISSN={0022-1236},
     journal={Journal of Functional Analysis},
      volume={290},
      number={3},
       pages={111249},
}

\bib{Christoffel1865}{article}{
      author={Christoffel, E.~B.},
       title={Ueber die {B}estimmung der {G}estalt einer krummen
  {O}berfl\"{a}che durch lokale {M}essungen auf derselben},
        date={1865},
        ISSN={0075-4102},
     journal={J. Reine Angew. Math.},
      volume={64},
       pages={193\ndash 209},
         url={https://doi.org/10.1515/crll.1865.64.193},
      review={\MR{1579295}},
}

\bib{Colesanti2025}{article}{
      author={Colesanti, A.},
      author={Focardi, M.},
      author={Guan, P.},
      author={Salani, P.},
       title={A constant rank theorem for linear elliptic equations on the
  sphere with applications to the mixed {C}hristoffel problem},
        date={2025},
      eprint={arXiv:2512.08670},
         url={https://arxiv.org/abs/2512.08670},
}

\bib{Colesanti2022a}{article}{
      author={Colesanti, Andrea},
      author={Ludwig, Monika},
      author={Mussnig, Fabian},
       title={The {H}adwiger theorem on convex functions, {III}: {S}teiner
  formulas and mixed {M}onge-{A}mp\`ere measures},
        date={2022},
        ISSN={0944-2669,1432-0835},
     journal={Calc. Var. Partial Differential Equations},
      volume={61},
      number={5},
       pages={Paper No. 181, 37},
         url={https://doi.org/10.1007/s00526-022-02288-3},
      review={\MR{4453228}},
}

\bib{Firey1967}{article}{
      author={Firey, W.~J.},
       title={The determination of convex bodies from their mean radius of
  curvature functions},
        date={1967},
        ISSN={0025-5793},
     journal={Mathematika},
      volume={14},
       pages={1\ndash 13},
         url={https://doi.org/10.1112/S0025579300007956},
      review={\MR{217699}},
}

\bib{Firey1968}{article}{
      author={Firey, William~J.},
       title={Christoffel's problem for general convex bodies},
        date={1968},
        ISSN={0025-5793},
     journal={Mathematika},
      volume={15},
       pages={7\ndash 21},
         url={https://doi.org/10.1112/S0025579300002321},
      review={\MR{230259}},
}

\bib{Firey1970}{article}{
      author={Firey, William~J.},
       title={Intermediate {C}hristoffel-{M}inkowski problems for figures of
  revolution},
        date={1970},
        ISSN={0021-2172},
     journal={Israel J. Math.},
      volume={8},
       pages={384\ndash 390},
         url={https://doi.org/10.1007/BF02798684},
      review={\MR{271838}},
}

\bib{Gardner2006}{book}{
      author={Gardner, Richard~J.},
       title={Geometric tomography},
     edition={Second},
      series={Encyclopedia of Mathematics and its Applications},
   publisher={Cambridge University Press, New York},
        date={2006},
      volume={58},
        ISBN={0-521; 0-521-68493-5},
         url={https://doi.org/10.1017/CBO9781107341029},
      review={\MR{2251886}},
}

\bib{Guan2003}{article}{
      author={Guan, Pengfei},
      author={Ma, Xi-Nan},
       title={The {C}hristoffel-{M}inkowski problem. {I}. {C}onvexity of
  solutions of a {H}essian equation},
        date={2003},
        ISSN={0020-9910,1432-1297},
     journal={Invent. Math.},
      volume={151},
      number={3},
       pages={553\ndash 577},
         url={https://doi.org/10.1007/s00222-002-0259-2},
      review={\MR{1961338}},
}

\bib{Heuser2006}{book}{
      author={Heuser, Harro},
       title={Gew\"ohnliche {D}ifferentialgleichungen},
     edition={Fifth},
      series={Mathematische Leitf\"aden. [Mathematical Textbooks]},
   publisher={B. G. Teubner, Stuttgart},
        date={2006},
        ISBN={3-519-42227-1},
        note={Einf\"uhrung in Lehre und Gebrauch. [Introduction to theory and
  application]},
      review={\MR{2382060}},
}

\bib{Hug2024}{article}{
      author={Hug, Daniel},
      author={Mussnig, Fabian},
      author={Ulivelli, Jacopo},
       title={Kubota-type formulas and supports of mixed measures},
        date={2024},
      eprint={arXiv:2401.16371},
}

\bib{Hug2024a}{article}{
      author={Hug, Daniel},
      author={Mussnig, Fabian},
      author={Ulivelli, Jacopo},
       title={Additive kinematic formulas for convex functions},
        date={2025},
     journal={Canadian Journal of Mathematics},
       pages={1–23},
}

\bib{Li2021}{article}{
      author={Li, Qi-Rui},
      author={Wan, Dongrui},
      author={Wang, Xu-Jia},
       title={The {C}hristoffel problem by the fundamental solution of the
  {L}aplace equation},
        date={2021},
        ISSN={1674-7283,1869-1862},
     journal={Sci. China Math.},
      volume={64},
      number={7},
       pages={1599\ndash 1612},
         url={https://doi.org/10.1007/s11425-019-1674-1},
      review={\MR{4280371}},
}

\bib{Minkowski1903}{article}{
      author={Minkowski, Hermann},
       title={Volumen und {O}berfl\"ache},
        date={1903},
        ISSN={0025-5831,1432-1807},
     journal={Math. Ann.},
      volume={57},
      number={4},
       pages={447\ndash 495},
         url={https://doi.org/10.1007/BF01445180},
      review={\MR{1511220}},
}

\bib{Mueller1966}{book}{
      author={M\"uller, Claus},
       title={Spherical harmonics},
      series={Lecture Notes in Mathematics},
   publisher={Springer-Verlag, Berlin-New York},
        date={1966},
      volume={17},
      review={\MR{199449}},
}

\bib{Mussnig2025}{article}{
      author={Mussnig, Fabian},
      author={Ulivelli, Jacopo},
       title={Explicit solutions to {C}hristoffel-{M}inkowski problems and
  {H}essian equations under rotational symmetries},
        date={2025},
      eprint={arXiv:2508.11600},
         url={https://arxiv.org/abs/2508.11600},
}

\bib{ONeill1983}{book}{
      author={O'Neill, Barrett},
       title={Semi-{R}iemannian geometry},
      series={Pure and Applied Mathematics},
   publisher={Academic Press, Inc. [Harcourt Brace Jovanovich, Publishers], New
  York},
        date={1983},
      volume={103},
        ISBN={0-12-526740-1},
        note={With applications to relativity},
      review={\MR{719023}},
}

\bib{Pogorelov1953}{article}{
      author={Pogorelov, A.~V.},
       title={On existence of a convex surface with a given sum of the
  principal radii of curvature},
        date={1953},
     journal={Uspehi Matem. Nauk (N.S.)},
      volume={8},
      number={3(55)},
       pages={127\ndash 130},
      review={\MR{57560}},
}

\bib{Schneider1977b}{article}{
      author={Schneider, Rolf},
       title={Das {C}hristoffel-{P}roblem f\"ur {P}olytope},
        date={1977},
     journal={Geometriae Dedicata},
      volume={6},
      number={1},
       pages={81\ndash 85},
         url={https://doi.org/10.1007/BF00181582},
      review={\MR{470865}},
}

\bib{Schneider2014}{book}{
      author={Schneider, Rolf},
       title={Convex bodies: the {B}runn-{M}inkowski theory},
     edition={expanded},
      series={Encyclopedia of Mathematics and its Applications},
   publisher={Cambridge University Press, Cambridge},
        date={2014},
      volume={151},
        ISBN={978-1-107-60101-7},
      review={\MR{3155183}},
}

\bib{Shenfeld2022}{article}{
      author={Shenfeld, Yair},
      author={van Handel, Ramon},
       title={The extremals of {M}inkowski's quadratic inequality},
        date={2022},
        ISSN={0012-7094,1547-7398},
     journal={Duke Math. J.},
      volume={171},
      number={4},
       pages={957\ndash 1027},
         url={https://doi.org/10.1215/00127094-2021-0033},
      review={\MR{4393790}},
}

\end{biblist}
\end{bibdiv}
	
	\endgroup

\end{document}